\documentclass[a4paper,12pt,leqno]{article}
\usepackage{amsmath}
\usepackage{amsthm}
\usepackage{amssymb}
\usepackage{amscd}
\usepackage{graphicx, color}
\usepackage{a4wide}
\usepackage[all]{xy}
\newtheorem{thm}{Theorem}[section]
\newtheorem{dfn}[thm]{Definition}
\newtheorem{prp}[thm]{Proposition}%[section]
\newtheorem{lmm}[thm]{Lemma}%[section]
\newtheorem{crl}[thm]{Corollary}%[section]
\newtheorem{example}[thm]{Example}%[section]
\newtheorem{rmk}[thm]{Remark}%[section]
\numberwithin{equation}{section}
\theoremstyle{remark}
\usepackage[all]{xy}

\newcommand{\Hol}{\mbox{{\rm Hol}}}

\newcommand{\Z}{\Bbb Z}
\newcommand{\C}{\Bbb C}
\newcommand{\R}{\Bbb R}

\newcommand{\T}{\Bbb T}
\renewcommand{\P}{{\rm Po}}
\newcommand{\SP}{\mbox{{\rm SP}}}

\newcommand{\Map}{\mbox{{\rm Map}}}

\newcommand{\CP}{\Bbb C {\rm P}}

\newcommand{\dis}{\displaystyle}
\newcommand{\p}{\prime}

\newcommand{\SZ}{{\mathcal{X}}^{D}}
\newcommand{\SZd}{{\mathcal{X}}^{D+\textit{\textbf{a}}}}
\newcommand{\I}{\mbox{{\rm (i)}}}
\newcommand{\II}{\mbox{{\rm (ii)}}}
\newcommand{\III}{\mbox{{\rm (iii)}}}
\newcommand{\IV}{\mbox{{\rm (iv)}}}

\newcommand{\XS}{X_{\Sigma}}

\newcommand{\GS}{G_{\Sigma}}

\newcommand{\dmin}{d_{\rm min}}
\newcommand{\rmin}{r_{\rm min}}
\newcommand{\KS}{\mathcal{K}_{\Sigma}}
\title{\bf
The homotopy type of
spaces of rational curves on a toric variety}
\author{Andrzej Kozlowski\footnote{%
Institute of Applied Mathematics and Mechanics,
University of Warsaw, Banacha 2, 02-097 Warsaw, Poland
(E-mail: akoz@mimuw.edu.pl)
}
\  and \ 
Kohhei Yamaguchi\footnote{%
Department of Mathematics,
University of Electro-Communications,  Chofu, Tokyo 182-8585, Japan
(E-mail: kohhe@im.uec.ac.jp);
The second author is supported by 
JSPS KAKENHI Grant Number 26400083.
}}

\date{}
\begin{document}
\maketitle
%%

%%%(Abstract)%%%%%.
\begin{abstract}
Spaces of holomorphic maps 
from the Riemann sphere
to  various complex manifolds (holomorphic curves )
have played an important role in several area of mathematics.
In a seminal paper G. Segal  investigated the homotopy type of holomorphic curves on complex projective spaces and M. Guest on  compact smooth toric varieties..
Recently Mostovoy and Villanueva, obtained a far reaching generalisation of these results, and in particular (for holomorphic curves)  improved the stability dimension obtained by Guest. 
In this paper, we    
generalize  their result  to holomorphic curves, 
on certain  non-compact smooth toric varieties. 
\end{abstract}
%%%%%%%%%%%%%%%%%%%%%

%%%%(SECTION 1: Introduction)%%%%
%%%%%
\section{Introduction}\label{section 1}
%%%%%
%%%%%
\paragraph{The motivation for this paper.}
%%%
For a complex manifold $X$, let 
$\Map^* (S^2,X)$ (resp. $\Hol^*(S^2,X)$)
denote the space of all base point preserving continuous maps
(resp. base point preserving holomorphic maps) 
from the Riemann sphere $S^2$ 
to $X$.
%If $m=1$, we can identify $\CP^1$ as the Riemann surface $S^2$
%and %in this case
The relation between the topology
of the space   $\Hol^* (S^2,X)$ and that of the space $\Map^* (S^2,X)$
has long been an object of study in several areas of mathematics and 
physics
(e.g. \cite{AJ}, \cite{A}).
Since Segal's  seminal study \cite{Se} of the case $X=\CP^n$, 
a number of mathematicians have investigated this and various closely related problems.
In particular, 
M. Guest \cite{Gu2} obtained the partial generalization of Segal's result
to the case of smooth compact toric varieties $X$. 
More recently,  J. Mostovoy and
E. Munguia-Villanueva \cite{MV}  proved a far-reaching generalization of Guest's result for the case of spaces of holomorphic maps from $\CP^{m}$ to a compact toric variety $X$ for  $m\geq 1$. The homology stability dimension which they obtained is also an improvement on Guest's result for the case $m=1$. 

In \cite{KY6}, the present authors studied this problem of the case $m=1$
%(for maps from $S^2=\CP^1$ only)  
%(also for $m=1$) and
for 
a certain family of non-compact
smooth toric subvarieties $X_I$ of $\CP^{n}$, and showed that the result of Mostovoy-Villanueva \cite{MV} can be extended to 
this case (with the same stability dimension).
\par\vspace{1mm}\par
In this paper, we  shall
prove that this result can be further extended to 
{\it non-compact} smooth toric varieties $X$ which satisfy certain two conditions
(see the conditions (\ref{equ: homogenous}.1) and (\ref{equ: homogenous}.2)). These conditions are satisfied for a wide range of smooth toric varieties (including all compact ones). 

 In fact, we will do better and show that under the 
 certain condition \lq\lq homology equivalence\rq\rq can be replaced by \lq\lq homotopy equivalence\rq\rq (up to the same dimension). 
%%%
%%% 
\par\vspace{1mm}\par
%%%%
The broad outline of our argument is analogous to Segal's seminal paper \cite{Se} (a brief sketch of such an argument is given in \cite[\S 5]{Gu2}).
%at the end of \S 5 of \cite{Gu2}). 
Namely, for a smooth toric variety $X$, we first prove  that there is a homotopy equivalence between certain limits of spaces $\Hol_D^*(S^2,X)$ of holomorphic maps, stabilized with respect to a suitably defined degree $D$, 
and the double loop space 
$\Map^*(S^2,X)=\Omega^2X$.  We can refer to this as 
{\it the stable result}.  The method used to prove it is a generalization of the {\it  scanning map}  technique used by Segal in \cite{Se}. In particular, we describe a generalization to the case of toric varieties of a fibration sequence
that plays the key role in Segal's argument
(see Proposition \ref{prp: Segal fibration}).
%%%
\par
Note that 
in  \cite{MV} a quite different stabilization 
is used, which is based on 
the Stone-Weierstrass theorem.
This stabilization has the advantage that it can used in the case of holomorphic maps 
from $\CP^m$ to a compact toric variety $X$ for any $m\ge 1$. 
However, the usefulness of the Stone-Weierstrass theorem is based on the fact that two holomorphic maps that are \lq uniformly close\rq, with respect to some metric, are actually homotopic. 
This is true when the metric on X is complete (e.g. when $X$ is compact), but not for general $X$. We are, therefore, unsure if our results can be extended to the case $\Hol^*(\CP^m,X)$, for $m>1$. Even if this is possible, we believe that our generalization of Segal's argument to the case of toric varieties is of some independent interest.
\par
%%%%%%%%%% 
The second part of the paper is concerned with establishing  \lq\lq 
homology stability dimensions\rq\rq\  for the inclusion map from  $\Hol_D^*(S^2,X)$ 
to the  double loop space $\Omega^2_DX$ of maps of degree $D$. 
These stability dimensions depend both on the degree $D$ and the toric variety $X$. The method is based on an a modification of the Vassiiev spectral sequence \cite{Va} due to Mostovoy  (\cite{Mo3}, \cite{Mo2}). The stability dimensions in homology are obtained by identifying a stable region of this spectral sequence. 
By observing that under a certain condition (described later) these mapping spaces  are simply connected, we can strengthen our  results by replacing homology equivalences by homotopy equivalences (up to the same dimension).

\paragraph{Basic definitions and notations.}
{\it A convex rational polyhedral cone} $\sigma$ 
in $\R^n$
is a subset of $\R^n$ of the form
%%%()%%
\begin{equation*}
\sigma =%\mbox{Cone}(S)=
\mbox{Cone}(\textit{\textbf{m}}_1,\cdots,\textit{\textbf{m}}_s)
=
\{\sum_{k=1}^s\lambda_k\textit{\textbf{m}}_k:\lambda_k\geq 0
\mbox{ for any }1\leq k\leq s\}
\end{equation*}
%%%
for some finite set $S=\{\textit{\textbf{m}}_k\}_{k=1}^s\subset\Z^n$.
The dimension of $\sigma$ is the dimension of
the smallest subspace which contains $\sigma$.
%%%%
A convex rational polyhedral cone $\sigma$
is called
{\it  strongly convex} if
$\sigma \cap (-\sigma)=\{{\bf 0}\}$.
%%
%%%
{\it A face} $\tau$ of 
%a strongly convex rational polyhedral cone
$\sigma$ is a subset $\tau\subset \sigma$ of the form
$\tau =\sigma\cap \{\textit{\textbf{x}}\in\R^n:L(x)=0\}$ 
for some linear form $L$ on $\R^n$, such that
$\sigma \subset \{\textit{\textbf{x}}\in\R^n:
L(\textit{\textbf{x}})\geq 0\}.$
In this case,
if we set 
$\{k:L(\textit{\textbf{m}}_k)=0\}=\{i_1,\cdots ,i_s\},$
we easily see that $\tau =\mbox{Cone}(\textit{\textbf{m}}_{i_1},
\cdots , \textit{\textbf{m}}_{i_s})$ and so that
a face $\tau$ is also a 
strongly convex rational polyhedral cone.
\par
A finite collection $\Sigma$ of strongly convex rational polyhedral cones
in $\R^n$ 
is called {\it a fan} in $\R^n$ if 
every face of an element of $\Sigma$ belongs to $\Sigma$ and
the intersection of any two elements of $\Sigma$ is a face of each.
%%%%%
\par
An $n$ dimensional irreducible normal  variety
$X$ (over $\C$) is called {\it a toric variety}
if it has a Zariski open subset
 $\T^n_{\C}=(\C^*)^n$ and the action of $\T^n_{\C}$ on itself
extends to an action of $\T^n_{\C}$ on $X$.
The most significant property of a toric variety
is the fact that it is characterized up to isomorphism entirely by its 
associated fan 
$\Sigma$. 
We denote by $\XS$ the toric variety associated to a fan $\Sigma$
(see \cite{CLS} in detail).
\par
It is well known that
there are no holomorphic maps
$\CP^1=S^2\to \T^n_{\C}$ except the constant maps, and that
the fan $\Sigma$ of $\T^n_{\C}$ is $\Sigma =\{{\bf 0}\}$.
Hence, without loss of generality
we always assume that $\XS\not=\T^n_{\C}$ and that
any fan $\Sigma$ in $\R^n$
satisfies the condition
%%%()%%
%\begin{equation}\label{eq: assumption0}
$\{{\bf 0}\}\subsetneqq \Sigma .$
%\end{equation}
%%%%
%%%%%
%%%
%%(Definition 1.1)%%
\begin{dfn}\label{dfn: fan}
%%%%%%
{\rm
For such a fan $\Sigma$ in $\R^n$,
%\begin{equation}\label{eq: assumption0}
let
%(1.1)%%
\begin{equation}\label{eq: one dim cone}
%%%%%%%
\Sigma (1)=\{\rho_1,\cdots ,\rho_r\}
\end{equation}
%%%%%
%where $r\geq 1$ is a positive integer, 
denote the set of all
one dimensional cones in $\Sigma$ for some positive integer $r$.
%Note that $r\geq 1$ is a positive integer under the condition (\ref{eq: assumption0}).
%%
For each integer $1\leq k\leq r$,
we denote by $\textbf{\textit{n}}_k\in\Z^n$ 
\textit{the primitive generator} of
$\rho_k$, such that %that is, we have
%%()%%
%\begin{equation}\label{eq: primitive}
%%%
$\rho_k \cap \Z^n=\Z_{\geq 0}\cdot \textbf{\textit{n}}_k.$
%\end{equation}
%%%%%%%%
%
Note that
$\rho_k =\mbox{Cone}(\textit{\textbf{n}}_k)=\R_{\geq 0}\cdot \textit{\textbf{n}}_k$.
}
\end{dfn}
%%%(End of Definition 1.1)%%%%

%%%(Definition 1.2)%%%
\begin{dfn}
{\rm
Let $K$ be a simplicial complex on the index set $[r]=\{1,2,\cdots ,r\}$,%
%%%%%%%%%%%%%%%%%%
%%%(Footnote 1)%%%%%%
\footnote{%
Let $K$ be some set of subsets of $[r]$.
Then the set $K$ is called {\it an abstract simplicial complex} on the index set $[r]$ if
the following condition holds:
 if $\tau \subset \sigma$ and $\sigma\in K$, then $\tau\in K$.
%\newline
In this paper by a simplicial complex $K$ we always mean an  \textit{an abstract simplicial complex}, 
and we always assume that a simplicial complex $K$  contains the empty set 
$\emptyset$.
}
%%%(End of FootNote 1)%%%%
%%%%%
and let $(\underline{X},\underline{A})=\{(X_1,A_1),\cdots ,(X_r,A_r)\}$ be a set of pairs of based spaces
such that $A_i\subset X_i$ for each $i$.
%%%
\par
(i)
 {\it The polyhedral product} $\mathcal{Z}_K
(\underline{X},\underline{A})$ of an $r$-tuple of pairs of spaces
$(\underline{X},\underline{A})$ with respect to $K$
is defined
by
$\dis\mathcal{Z}_K(\underline{X},\underline{A})
=
\bigcup_{\sigma\in K}(\underline{X},\underline{A})^{\sigma},$
where we set
%%%(1.2)%%
\begin{equation}
%%%
(\underline{X},\underline{A})^{\sigma}
=
\{(x_1,\cdots ,x_r)\in X_1\times \cdots \times X_r:
x_k\in A_k\mbox{ if }k\notin \sigma\}.
\end{equation}
When $(X_i,A_i)=(X,A)$ for each $1\leq i\leq r$, we write
$\mathcal{Z}_K(X,A)=\mathcal{Z}_K(\underline{X},\underline{A}).$
%%%%
\par
(ii)
For each subset $\sigma =\{i_1,\cdots ,i_s\}\subset [r]$, let
$L_{\sigma}$ denote {\it the coordinate subspace} of $\C^r$ defined by
%%(1.3)%%
\begin{equation}
%%%%%%%
L_{\sigma}=\{(x_1,\cdots ,x_r)\in\C^r:
x_{i_1}=\cdots =x_{i_s}=0\}.
\end{equation}
%%%%%%
Let $U(K)$ denote {\it the complement of coordinate subspaces of type} $K$
given by
%%%(1.4)%%%
\begin{equation}
U(K)=
\C^r\setminus \bigcup_{\sigma\in I(K)}L_{\sigma}
=\C^r\setminus \bigcup_{\sigma\subset [r],\sigma\notin K}L_{\sigma},
%%%
\end{equation}
%%%%%%
where we set
%%%%(1.5)%%
\begin{equation}\label{eq: IK}
%%%%%%%%%%
I(K)=\{\sigma\subset [r]:\sigma\notin K\}.
\end{equation}
%%%%%%
Note that $U(K)$ is the Alexander dual of the space
$L(\Sigma)=\bigcup_{\sigma\in I(K)}L_{\sigma}$
in $\C^r$, and it is easy to see that
%%(1.6)%%
\begin{equation}\label{equ: UK}
%%%%%%%%%%
U(K)=\mathcal{Z}_K(\C,\C^*).
\end{equation}
%%%%%%
%%%%%
\par
%%%(iii)%%%%
(iii)
%%%%%%%%%%%
For a fan $\Sigma$ in $\R^n$ as in Definition \ref{dfn: fan},
let $\mathcal{K}_{\Sigma}$ denote {\it the underlying simplicial complex of}  $\Sigma$
defined by
%%%%%
%%%(1.7)%%%%%%
\begin{equation}
%%%%%%%%%%%%%
\KS =
\Big\{\{i_1,\cdots ,i_s\}\subset [r]:
\textbf{\textit{n}}_{i_1},\textbf{\textit{n}}_{i_2},\cdots
,\textbf{\textit{n}}_{i_s}
\mbox{ span a cone in }\Sigma\Big\}.%\cup \{\emptyset\}.
\end{equation}
%%%%%%
It is easy to see that $\KS$ is a simplicial complex on the index set $[r]$.
}
\end{dfn}
%%%(end of Definition 1.2)%%%%%
%%
%%%
%%%(Remark 1.3)%%%
\begin{rmk}\label{rmk: fan}
%%%%%%%
{\rm
The fan $\Sigma$ is completely determined by
the pair
$(\mathcal{K}_{\Sigma},\{\textit{\textbf{n}}_k\}_{k=1}^r).$
In fact, if we set 
%%%%%%%%%
%%()%%
%\begin{equation}
$\mbox{C}(\sigma)=
%\begin{cases}
\mbox{Cone}
(\textit{\textbf{n}}_{i_1},\cdots ,\textit{\textbf{n}}_{i_s})$
if
$\sigma =\{i_1,\cdots ,i_s\}\in \mathcal{K}_{\Sigma}$
and $C(\emptyset)=\{{\bf 0}\}$,
%%%%
then
it is easy to see that
%%()%%
%\begin{equation}
$\Sigma =
\{\mbox{C}(\sigma):\sigma \in \mathcal{K}_{\Sigma}\}.$
%\end{equation}
%%%
\qed
%%%%
}
%%%%%%%
\end{rmk}
%%%%%%%(End of Remark 1.3)%%%%%%%%

%%%%%%%%%%%
%%%(Definition 1.4)%%%
\begin{dfn}[\cite{BP}, Definition 6.27, Example 6.39]
%%%
{\rm
Let $K$ be a simplicial complex on the index set $[r]$.
Then we denote by $\mathcal{Z}_K$ and $DJ(K)$
{\it the moment-angle complex} of $K$ and 
{\it the Davis-Januszkiewicz space} of $K$, 
respectively,
which are
defined by}
%%(1.8)%%%
\begin{equation}\label{DJ}
%%%%%%%%
\mathcal{Z}_K=\mathcal{Z}_K(D^2,S^1),\quad
DJ(K)=\mathcal{Z}_K(\CP^{\infty},*).
%%%
\end{equation}
%%%%%
\end{dfn}
%%(End of Example 1.3)%%%

%%%(Definition 1.5)%%
\begin{dfn}
%%%%%%%
{\rm
Let $\Sigma$ be a fan in $\R^n$ as in Definition \ref{dfn: fan}.
%%%
Let $\GS\subset \T^r_{\C}=(\C^*)^r$ denote the subgroup defined by
%%%%%%%%
%%(1.9)%%
\begin{equation}
%%%%%%%%
\GS =\{(\mu_1,\cdots ,\mu_r)\in \T^r_{\C}:
\prod_{k=1}^r(\mu_k)^{\langle \textbf{\textit{n}}_k,\textbf{\textit{m}}
\rangle}=1
\mbox{ for any }\textbf{\textit{m}}\in\Z^n\},
%%%%%%%%%%%%
\end{equation}
%%%
where
$\langle \textbf{\textit{u}}, \textbf{\textit{v}}\rangle=
\sum_{k=1}^nu_kv_k$ for
$\textbf{\textit{u}}=(u_1,\cdots ,u_n)$
and  $\textbf{\textit{v}}=(v_1,\cdots ,v_n)\in\R^n$.
%%%%
\par
Then consider the natural $\GS$-action on
$\mathcal{Z}_{\mathcal{K}_{\Sigma}}(\C,\C^*)$ given by 
coordinate-wise multiplication, i.e.
%%%%%%%%%
%%()%%
%\begin{equation}\label{eq: multiplication}
%%%%%%%%%%
$
(\mu_1,\cdots ,\mu_r)\cdot(x_1,\cdots ,x_r)=
(\mu_1x_1,\cdots ,\mu_rx_r).
$
%%%%%%%%%%
\par
Let 
$\mathcal{Z}_{\mathcal{K}_{\Sigma}}(\C,\C^*)/\GS =U(\KS)/\GS$
denote the corresponding orbit space.
}
\end{dfn}
%%(End of Definition 1.5)%%%%

%%%(Theorem 1.6: Theorem of Cox)%%
\begin{thm}
[\cite{Cox1}, Theorem 2.1; \cite{Cox2}, Theorem 3.1]\label{prp: Cox}
%%%%%%%%%
Suppose that the set $\{\textit{\textbf{n}}_k\}_{k=1}^r$ 
of all primitive generators  spans $\R^n$
$($i.e. $\sum_{k=1}^r\R\cdot \textit{\textbf{n}}_k =\R^n).$
\par
%%(i)%%%
$\I$
Then
there is a natural isomorphism
%%(1.10)%%
\begin{equation}\label{equ: homogenous1}
\XS\cong
\mathcal{Z}_{\mathcal{K}_{\Sigma}}(\C,\C^*)/\GS =U(\KS)/\GS . 
\end{equation}
%%%%
%%%
\par
%%(ii)%%%%%%
$\II$
%%%
If
$f:\CP^m\to \XS$ is a holomorphic map, %then
there exists an $r$-tuple $D=(d_1,\cdots ,d_r)\in (\Z_{\geq 0})^r$
of non-negative integers
satisfying the condition
$\sum_{k=1}^rd_k\textit{\textbf{n}}_k={\bf 0}$
and homogenous polynomials $f_i\in \C [z_0,\cdots ,z_m]$ of degree $d_i$
$(i=1,2,\cdots, r)$
such that 
polynomials $\{f_{i}\}_{i\in\sigma}$ have no common root except
${\bf 0}\in\C^{m+1}$ for each $\sigma \in I(\mathcal{K}_{\Sigma})$
and that
the diagram
%%(1.11)%%
\begin{equation}\label{eq: homogenous-CD}
%%%%
\begin{CD}
\C^{m+1}\setminus \{{\bf 0}\} @>(f_1,\cdots ,f_r)>> U(\mathcal{K}_{\Sigma})
\\
@V{\gamma_m}VV @V{q_{\Sigma}}VV
\\
\CP^m @>f>> U(\mathcal{K}_{\Sigma})/\GS =\XS
\end{CD}
\end{equation}
%%%%%%%%
is commutative,
where $\gamma_m:\C^{m+1}\setminus \{{\bf 0}\}\to \CP^m$ denotes
the canonical Hopf fibering and the map $q_{\Sigma}$ is a canonical projection
induced from the identification $($\ref{equ: homogenous1}$)$.
In this case, we call this holomorphic map $f$ as a holomorphic map of degree $D=(d_1,\cdots ,d_r)$ and we
 represent it  as
%%(1.12)%%%
\begin{equation}\label{equ: homogenous}
f=[f_1,\cdots ,f_r].
\end{equation}
%%%%%%%%%%%%
\par
%%%(iii)%%%%
$\III$
If $g_i\in \C [z_0,\cdots ,z_m]$ is a homogenous polynomial
of degree $d_i$
$(1\leq i\leq r)$ such that $f=[f_1,\cdots ,f_r]=[g_1,\cdots ,g_r]$,
there  exists some element $(\mu_1,\cdots ,\mu_r)\in \GS$ such that
$f_i=\mu_i\cdot g_i$ for each $1\leq i\leq r$.
Thus, such $r$-tuple
$(f_1,\cdots ,f_r)$ of homogenous polynomials representing the holomorphic map $f$
is uniquely determined up to
$\GS$-action.
\qed
%%%
\end{thm}
%%%%%%(End of Proposition 1.6)%%%%%

%%%%(Assumptions)%%%%%%%
\paragraph{Assumptions.}
%%%%%%
From now on, let $\Sigma$ be a fan in $\R^n$
satisfying the condition (\ref{eq: one dim cone})
as in Definition \ref{dfn: fan}, 
and
we shall assume  that the following two conditions hold.
%%%%(Assumptions)%%%%
\begin{enumerate}
%%(1.12.1)%%%%
\item[(\ref{equ: homogenous}.1)]
%%%%%%%%%%%%%%%
%The condition (\ref{eq: assumption0}) holds and
The set
$\{\textbf{\textit{n}}_k\}_{k=1}^r$ of primitive generators spans $\Z^n$
over $\Z$, i.e.
\newline
$\sum_{k=1}^r\Z\cdot \textit{\textbf{n}}_k=\Z^n$.
%%(1.12.2)%%
\item[(\ref{equ: homogenous}.2)]
%%%%
There is an $r$-tuple $D=(d_1,\cdots ,d_r)\in (\Z_{\geq 1})^r$ 
%of positive integers 
such that
$\sum_{k=1}^rd_k\textit{\textbf{n}}_k={\bf 0}.$
%%%
\end{enumerate}
%%%%%

%%%(Remark 1.7)%%%
\begin{rmk}\label{rmk: assumption}
%%%%%%%%%%%%%%%%%%
{\rm
(i)
Note that  the condition (\ref{equ: homogenous}.1) always holds
if $\XS$ is a compact smooth toric variety
(by Lemma \ref{lmm: toric} below).
\par
(ii)
If the condition (\ref{equ: homogenous}.1) holds,
one can easily see that  the set $\{\textit{\textbf{n}}_k\}_{k=1}^r$ spans
$\R^n$ over $\R$
and we have the  isomorphism (\ref{equ: homogenous1}).
%%%%%
%}
%%%%%%%%%%%%%%%%%%
%\end{rmk}
%%%(End of Remark 1.6)%%%
%
%
%%%(Remark)%%%
%\begin{rmk}\label{rmk: constant map}
%%%%%
%{\rm
\par
(iii)
Let $\Sigma$ denote the fan in $\R^2$ given by
$\Sigma =\{{\bf 0},\mbox{Cone}(\textit{\textbf{n}}_1),
\mbox{Cone}(\textit{\textbf{n}}_2)\}$
for the standard basis $\textit{\textbf{n}}_1=(1,0),$
$\textit{\textbf{n}}_2=(0,1)$.
Then the toric variety $\XS$ of  $\Sigma$ is
$\C^2$ which has trivial homogenous coordinates.
It is clearly a (simply connected) smooth toric variety, and
the condition (\ref{equ: homogenous}.1) also  holds.
%\par
However, in this case,
$\sum_{k=1}^2d_k\textit{\textbf{n}}_k={\bf 0}$ iff
$(d_1,d_2)=(0,0)$. 
Hence, it follows from (ii) of Proposition \ref{prp: Cox} that there are no holomorphic maps
$\CP^1=S^2\to \XS=\C^2$ other than the constant maps.
Assuming the condition
(\ref{equ: homogenous}.2) guarantees the existence of non-trivial
holomorphic maps. Of course, it would be sufficient to assume that 
$(d_1,\dots ,d_r)\not= (0,\dots 0)$ but if $d_i=0$ for some $i$, then the number $d(D,\Sigma)$ (defined in (1.17)) is not a positive integer
and our assertion (Theorem 1.9 below) is vacuous.
For this reason, we will assume the condition (\ref{equ: homogenous}.2).
\qed
}
\end{rmk}
%%(End of Example 1.7)%%
%%%%%
\paragraph{Spaces of holomorphic maps.}
%%%%%%%%
Let $\XS$ be a smooth toric variety and
 we make the identification $\XS =U(\KS)/\GS$. 
 Now
consider a base point preserving holomorphic map $f=[f_1,\cdots ,f_r]:
\CP^m\to \XS$ for the
case $m=1$.
In this situation,  we  identify $\CP^1=S^2=\C \cup \infty$ and choose the points $\infty$ and $[1,1,\cdots ,1]$ as the base points of
$\CP^1$ and $\XS$, respectively.
Then by taking $z=\frac{z_0}{z_1}$
we can view $f_k$ as a monic polynomial  $f_k(z)\in\C [z]$
of degree $d_k$ for each $1\leq k \leq r$ with the complex variable  $z$.
Now we can define the space of holomorphic maps as follows.

%%%%(Definition 1.8)%%%%%%
\begin{dfn}\label{dfn: holomorphic}
%%%%%%%%%%%%%%%%%%%%%%%%%%%
{\rm
(i)
Let $\P^d(\C)$ denote the space of all monic polynomials 
$f(z)=z^d+a_1z^{d-1}+\cdots +a_{d-1}z+a_d\in \C [z]$ of degree
$d$, and
we set
%%(1.13)%%%
\begin{equation}\label{eq: Pd}
%%%%%%%%
\P^D=\P^{d_1}(\C)\times \P^{d_2}(\C)\times\cdots \times \P^{d_r}(\C).
\end{equation}
%%%
Note that there is an homeomorphism
$\P^d(\C)\cong \C^d$ by identifying
$z^d+\sum_{k=1}^d a_kz^{d-k}
\mapsto (a_1,\cdots ,a_d)\in\C^d.$
\par
%%%(ii)%%%
(ii)
For $r$-tuple $D=(d_1,\cdots ,d_r)\in (\Z_{\geq 1})^r$ 
of positive integers
satisfying the condition (\ref{equ: homogenous}.2),
%%% 
let $\Hol_D^*(S^2,\XS)$ denote the
space of all $r$-tuples
$(f_1(z),\cdots ,f_r(z))\in \P^D$ of monic polynomials
satisfying the condition
%%%%%
\begin{enumerate}
%%()%%
\item[$(\dagger)$]
%%%%%%%
the polynomials
$f_{i_1}(z),\cdots ,f_{i_s}(z)$ have no common root
for any $\sigma =\{i_1,\cdots ,i_s\}\in I(\KS)$,
i.e.
$(f_{i_1}(\alpha),\cdots ,f_{i_s}(\alpha))\not= (0,\cdots ,0)$
for any $\alpha \in  \C$.
\end{enumerate}
%%%%%
By identifying 
$\XS =U(\KS)/\GS$ and $\CP^1=S^2=\C\cup\infty$,
define the natural inclusion map
%%%%%%%%
%%)%%
%%%%%%%%
$i_D:\Hol_D^*(S^2,\XS)\to \Map^*(S^2,\XS)=\Omega^2\XS$ by
%%%%%%
%%
%%(1.14)%%
\begin{equation}
%%%%%%%%%
i_D(f_1(z),\cdots ,f_r(z))(\alpha )=
\begin{cases}
[f_1(\alpha),\cdots ,f_r(\alpha)] & \mbox{ if }\alpha \in\C
\\
[1,1,\cdots ,1] & \mbox{ if }\alpha =\infty
\end{cases}
\end{equation}
%%%%%%%%%%
where
we choose the points $\infty$ and $[1,1,\cdots ,1]$
as the base points of $S^2$ and $\XS$, respectively.
%%
%\par
Since the representation of polynomials in $\P^D$ representing a 
base point preserving holomorphic map of degree $D$
is uniquely determined, 
the space $\Hol_D^*(S^2\XS)$ can be identified with the space of base point preserving holomorphic maps of 
degree $D$.
%\par
%%%
Moreover,
since $\Hol_D^*(S^2,\XS)$ is connected, the image of
$i_D$
is contained in a certain path-component of $\Omega^2\XS$,
which is denoted by $\Omega^2_D\XS$.
Then we have the natural inclusion map
%%%(1.15)%%%%%
\begin{equation}
%%%%%%%%
i_D:\Hol_D^*(S^2,\XS)\to \Map^*_D(S^2,\XS)=\Omega^2_D\XS .
\end{equation}
%%%%%%%
\par\vspace{1mm}\par
%%%%%
(iii)
We say that a set $\{\textbf{\textit{n}}_{i_1},\cdots ,\textbf{\textit{n}}_{i_s}\}$
%of some primitive generators
is a {\it a primitive collection} 
 if it does not span a cone in $\Sigma$ but any proper subset of it 
does.
\par
Then
define the integers $\rmin (\Sigma)$ and $d(D,\Sigma)$
by
%%(1.16)%%%%%%%
\begin{align}\label{eq: rmin}
%%%%%%%%%%%%%%%
\rmin (\Sigma)
&=\min\{s\in\Z_{\geq 1}:
\{\textbf{\textit{n}}_{i_1},\cdots ,\textbf{\textit{n}}_{i_s}\}
\mbox{ is a primitive collection}\},
\\
%%(1.17)%%%%%
d(D,\Sigma)&=(2\rmin  (\Sigma) -3)d_{\rm min}-2,
\mbox{ where }d_{\rm min}=\min \{d_1,\cdots ,d_r\}.
\end{align}
%%%%
%%%(iv)%%
\par
(iv) A map $f:X\to Y$ is {\it a homology equivalence through dimension} $N$
(resp. {\it a homotopy equivalence through dimension} $N$)
if the induced homomorphism
$
f_*:H_k(X,\Z) \to H_k(Y,\Z)$
$(\mbox{resp. }f_*:\pi_k(X)\to \pi_k(Y))
$
is an isomorphism for any $k\leq N$.
%%%%%%%
}
%%%%%%
\end{dfn}
%%%%%%%%(End of Definition 1.8)%%%
%\par\vspace{2mm}\par

\paragraph{The main results.}
The main result of this paper generalizes the results given in
\cite{KY6} and extend some result obtained in \cite{MV}
as follows.
%%%%%

%%%(Theorem 1.9: The main Theorem)%%%%%
\begin{thm}\label{thm: I}
%%% 
Let $\XS$ be a smooth toric variety such that 
the conditions $($\ref{equ: homogenous}.1$)$ and 
$($\ref{equ: homogenous}.2$)$ are satisfied.
Then
the inclusion map
$$
i_D:\Hol_D^*(S^2,\XS) \to \Omega^2_D\XS
$$ 
is a homotopy equivalence
through dimension $d(D,\Sigma)$ if $\rmin (\Sigma)\geq 3$
and a homology equivalence through dimension 
$d(D,\Sigma )=d_{min}-2$ if $\rmin (\Sigma)=2$.
%%%%
\end{thm}
%%%%%%%%(End od Theorem 1.9)%%

%%%(Remark 1.10)%%
\begin{rmk}
%%%
{\rm
(i)
If $\XS$ is compact,  we know that
the  map $i_D$ is a homology equivalence through dimension
$d(D,\Sigma)$ by the result of
Mostovoy-Villanueva \cite{MV}. 
However, their argument is based on the Stone-Weierstrass theorem which, as mentioned above, requires $\XS$ to be compact (or at least to possess a complete metric). 
\par
(ii)
If $\rmin (\Sigma)\geq 3$,  Theorem \ref{thm: I} states
that the map $i_D$ is a homotopy equivalence through
the dimension $d(D,\Sigma)$. So the assertion of Theorem \ref{thm: I}
is stronger than that of \cite{MV} even if $\XS$ is compact (for $m=1$).
Moreover,
we conjecture that the map $i_D$ is  a homotopy equivalence through the same dimension
even when $\rmin (\Sigma)=2$.
Although we cannot prove this, there are several reasons which support this conjecture
 (for example, see (ii) of Corollary \ref{crl: pi1}). In fact, the conjecture is known to hold for certain non-compact toric varieties $X_n$
(\cite[Theorem 1.6]{GKY1}).
\qed
}
\end{rmk}
%%%%(End of Remark 1.10)%%
%
%%(Corollary 1.11)%%
\begin{crl}\label{crl: I}
%%%%%%%%%%%%%%%%%
Let $\XS$ be a compact smooth toric variety and let
$D=(d_1,\cdots ,d_r)\in (\Z_{\geq 1})^r$
be an $r$-tuple of positive integers satisfying the condition
(\ref{equ: homogenous}.2).
Let $\Sigma (1)$ denote the set of all one dimensional cones in $\Sigma$, and $\Sigma_1$ any fan in $\R^n$ such that
%%%%()%%
%\begin{equation}\label{eq: Sigma}
%%%%%
$\Sigma (1)\subset \Sigma_1\subsetneqq \Sigma.$
%\end{equation}
%%%%%%%
Then $X_{\Sigma_1}$ is a non-compact smooth toric subvariety of $\XS$ and
the inclusion map
$$
i_D:\Hol_D^*(S^2,X_{\Sigma_1})\to \Omega^2_DX_{\Sigma_1}
$$ is a
homotopy equivalence through the dimension $d(D,\Sigma_1)$
if $\rmin (\Sigma_1)\geq 3$
and a homology equivalence through dimension 
$d(D,\Sigma_1 )=d_{min}-2$ if $\rmin (\Sigma_1)=2$.
%%%%%%%%%
\end{crl}
%%%%%%%%%%%%(End of Corollary 1.11)%%%%

Since the case $\XS=\CP^n$ of Corollary \ref{crl: I}
was treated in \cite{KY6}, we shall consider
the case that $\XS$ is the Hirzerbruch surface $H(k)$. 

%%%%(Example 1.12)%%
\begin{example}\label{dfn: Hirzerbruch}
%%%%%%%%%%%%%%%%%%
{\rm
For an integer $k\in \Z$, let
$H(k)$ be {\it the Hirzerbruch surface} defined by
%()%%
\begin{equation*}
H(k)=\{([x_0:x_1:x_2],[y_1:y_2])\in\CP^2\times \CP^1:x_1y_1^k=x_2y_2^y\}
\subset \CP^2\times \CP^1.
\end{equation*}
%%%
Since there are isomorphisms
$H(-k)\cong H(k)$ for $k \not=0$ and $H(0)\cong\CP^2\times \CP^1$,
without loss of generality we can assume that $k\geq 1$.
Let $\Sigma_k$ denote the fan
 in $\R^2$ given by
%%%%%%%
$$
\Sigma_k=
\big\{\mbox{Cone}({\bf n}_i,{\bf n}_{i+1})\ (1\leq i\leq 3),
\mbox{Cone}({\bf n}_4,{\bf n}_1),
\mbox{Cone}(\textit{\textbf{n}}_j)\ (1\leq j\leq 4),\ \{{\bf 0}\}\big\},
$$
where we set
$
\textit{\textbf{n}}_1%=\textit{\textbf{e}}_1
=(1,0),\  
\textit{\textbf{n}}_2%=\textit{\textbf{e}}_2
=(0,1),\  
\textit{\textbf{n}}_3%=-\textit{\textbf{e}}_1+k\textit{\textbf{e}}_2
=(-1,k), 
\ 
\textit{\textbf{n}}_4%=-\textit{\textbf{e}}_2
=(0,-1).
$
\par
It is easy to see that $\Sigma_k$ is the fan of $H(k)$
and that $\Sigma_k(1)=\{\mbox{Cone}(\textit{\textbf{n}}_i):1\leq i\leq 4\}$.
Since $\{\textit{\textbf{n}}_1,\textit{\textbf{n}}_3\}$ and
$\{\textit{\textbf{n}}_2,\textit{\textbf{n}}_4\}$ are the 
only primitive
collections, $\rmin (\Sigma_k)=2$.
%%%
%%
%\par
%%
Moreover, for a $4$-tuple $D=(d_1,d_2,d_3,d_4)\in (\Z_{\geq 1})^4$
of positive integers,
the equality
$\sum_{k=1}^4d_k\textit{\textbf{n}}_k=\textbf{0}$ holds 
if and only if
$(d_1,d_2,d_3,d_4)=(d_1,d_2,d_1,kd_1+d_2)$, and
$
\dmin =\min \{d_1,d_2,d_3,d_4\}
=\min\{d_1,d_2\}.
$
Hence, by Corollary \ref{crl: I}, we have the following:
%%%%%%%%%%%%%%%%%%%%%%%%
}
\end{example}
%%(End of Example 1.12)%%
%%

%%%%(Corollary 1.13)%%%
\begin{crl}\label{example: H(k)}
%%%%%%%%%%%%%%%%%%%%
Let $k\geq 1$ be a positive integer and 
$\Sigma$ be a fan in $\R^2$ such that
$
\Sigma_k(1)=\{\mbox{\rm Cone}(\textit{\textbf{n}}_i):1\leq i\leq 4\}
\subset \Sigma \subsetneqq \Sigma_k$
as in Example \ref{dfn: Hirzerbruch}.
Then 
$\XS$ is a non-compact open smooth subvariety of $H(k)$.
If $D=(d_1,d_2.d_1,kd_1+d_2)\in (\Z_{\geq 1})^4$,
the inclusion map
$$
i_D:\Hol_D^*(S^2,\XS)\to \Omega^2_D\XS
$$ is a homology equivalence
through dimension $\min \{d_1,d_2\}-2$.
\qed
%%%%
\end{crl}
%%(End of Corollary 1.13)%%
%
%%(Remark 1.14)%%
\begin{rmk}
%%%%
{\rm
(i) There are $15$ non isomorphic (as varieties) non-compact subvarieties 
$\XS$ of $H(k)$
which satisfy the assumption of Corollary \ref{example: H(k)}.
\par
(ii)
Note that there is an isomorphism
%%(1.24)%%
\begin{equation}
\pi_2(\XS)\cong \Z^{r-n}
\qquad
(\mbox{see Lemma \ref{lmm: XS} below)}, 
\end{equation}
%%%% 
in general.
%(see Lemma \ref{lmm: XS} below),
So $(r-n)$ of the $r$ positive integers $\{d_k\}_{k=1}^r$ can be 
chosen freely.
For example, in Example \ref{dfn: Hirzerbruch},
$(r,n)=(4,2)$ and
$r-n=4-2=2$. In this case, only two positive integers $d_1$, $d_2$ can be chosen freely
and the other integers $d_3$ and $d_4$ are determined by
uniquely as
$(d_3,d_4)=(d_1,kd_1+d_2).$
\qed
%%%%%%%%
}
\end{rmk}
%%(End of Remark 1.14)%%%

This paper is organized as follows.
In \S \ref{section: scanning map} we introduce the scanning map and
prove the stability theorem for the scanning map
(Theorem \ref{thm: scanning map}).
In \S \ref{section 3}
we recall the basic properties of polyhedral products,
determine the homotopy type of the space 
$E^{\Sigma}(\overline{U},\partial \overline{U})$ 
(Lemma \ref{lmm: E-infty})
and
prove the existence of a Segal-type fibration sequence
(Proposition \ref{prp: Segal fibration}).
In \S \ref{section: stability} we prove our main stability result
(Theorem \ref{thm: II}) by using
(Theorem \ref{thm: scanning map}).
In \S \ref{section: simplicial resolution} we recall
the notion of a simplicial resolution and
in \S \ref{section: spectral sequence}  construct the
(non-degenerate) Vassiliev spectral sequence and its truncated
spectral sequence for computing the homology of $\Hol_D^*(S^2,\XS)$.
Finally in \S \ref{section: proofs}  we prove an unstable stability result
(Theorem \ref{thm: III}) and use it to
prove our main results (Theorem \ref{thm: I}
and Corollary \ref{crl: I}).

%%%(SECTION 2)%%%%
%%%%%%%%%%%%%%%%%%%%%%%%%%%%%
%%%(SECTION 2)%%
\section{The scanning map}\label{section: scanning map}
%%%%%%%%%%%%%%%

First, we recall some known results.

%%%
%%%%%(Lemma 2.1)%%%
\begin{lmm}[\cite{BP}; Corollary 6.30, Theorem 6.33,
Theorem 8.9]\label{Lemma: BP}
%%%%%%%%%%%%%%%%%%%
Let $K$ be a simplicial complex on the index set $[r]$.
\par
$\I$
The space $\mathcal{Z}_K$ is $2$-connected, and
there is a fibration sequence
%%(2.1)%%
\begin{equation}
\mathcal{Z}_K \stackrel{}{\longrightarrow} DJ(K)
\stackrel{\subset}{\longrightarrow} (\CP^{\infty})^r.
\end{equation}
\par
$\II$
There is an $(S^1)^r$-equivariant deformation retraction
%%(2.2)%%
\begin{equation}\label{eq: retr}
ret:\mathcal{Z}_K(\C,\C^*)\stackrel{\simeq}{\longrightarrow}
\mathcal{Z}_K.
\qquad
\qed
\end{equation}
\end{lmm}
%%%%(End of Lemma 2.1)%%

%%(Lemma 2.2)%%%%%%
\begin{lmm}[\cite{Pa1}; (6.2) and Proposition 6.7]\label{lmm: principal}
%%%%%%%%%%%%%%%
Let $\XS$ be a smooth toric variety such that the condition 
(\ref{equ: homogenous}.1) holds.
Then
there is an isomorphism 
%%(2.3)%%
\begin{equation}\label{eq: GS-eq}
\GS\cong \T^{r-n}_{\C}=(\C^*)^{r-n},
\end{equation}
%%%%%
and
the group $\GS$ acts on the space $\mathcal{Z}_{\mathcal{K}_{\Sigma}}(\C,\C^*)$
freely and there is a principal
$\GS$-bundle
%%%
%%%(2.4)%%
\begin{equation}\label{eq: principal}
%%%%%%%%%
q_{\Sigma}:
U(\mathcal{K}_{\Sigma})=\mathcal{Z}_{\mathcal{K}_{\Sigma}}(\C,\C^*)\to\XS.
\qquad
\qed
%%%%%%%%%%%%%%
\end{equation}
%%%%%%%%%%%%%%}
%%%%%
\end{lmm}
%%%(End of Lemma 2.2)%%

%%%(Lemma 2.3)%%%
\begin{lmm}\label{lmm: XS}
%%%%%%%%%%%%%%%%%
If the condition (\ref{equ: homogenous}.1) is
satisfied, 
the space $\XS$ is simply connected and $\pi_2(\XS)=\Z^{r-n}$. 
%%%
\end{lmm}
%%%%%%%%%%%%%%%%%
\begin{proof}
%%%
By  (\ref{equ: homogenous}.1) and \cite[Theorem 12.1.10]{CLS}
we easily see that
the space $\XS$ is simply connected.
%%%
By Lemma \ref{Lemma: BP} and (\ref{eq: retr}),  $\mathcal{Z}_{\KS}(\C,\C^*)$ is $2$-connected.
Thus,
by using the homotopy exact sequence of the 
principal $\GS$-bundle (\ref{eq: principal}) and
(\ref{eq: GS-eq}), we   
see that  $\pi_2(\XS)= \Z^{r-n}$.
%%%%%%%
\end{proof}
%%%%%%%%%(End of Proof of Lemma 2.3)%%

Recall the basic facts concerning the relation between a fan and a toric variety.
%%%(Definition 2.4)%%
\begin{dfn}
%%%%%%%
{\rm
Let $\Sigma$ be a fan in $\R^n$.
Then a cone $\sigma\in \Sigma$ is called {\it smooth} if it is generated by
a subset of a basis of $\Z^n$.
%Let $\vert\Sigma\vert$ denote the support of $\Sigma$ defined by
%the subset
%$\vert\Sigma\vert =\bigcup_{\sigma\in\Sigma}\sigma$.
%%
}
%%%%%
\end{dfn}
%%(End of Definition 2.4)%%%

%%%(Lemma 2.5)%%
\begin{lmm}[\cite{CLS}]\label{lmm: toric}
%%%%%%%%%%%%%%%%
Let $\XS$ be a toric variety determined by a fan $\Sigma$ in $\R^n$.
\begin{enumerate}
\item[$\I$]
$\XS$ is a smooth if and only if every cone $\sigma\in\Sigma$ is smooth.
\item[$\II$]
$\XS$ is compact if and only if $\R^n=\bigcup_{\sigma\in \Sigma}\sigma$.
\qed
\end{enumerate}
%%%%%
\end{lmm}
%%%%%%%(End of Lemma 2.5)%%%
%

Now, we consider  configuration spaces the scanning map.

%%(Definition 2.6)%%%
\begin{dfn}
%%%%%%%%%%%%%%%%%%%%%
{\rm
 For a positive integer $d\geq 1$ and a based space $X$, let
 $\SP^d(X)$ denote {\it the $d$-th symmetric product} of $X$ defined by
the orbit space 
%%(2.5)%%
\begin{equation}
\SP^d(X)=X^d/S_d,
\end{equation}
where the symmetric group $S_d$ of $d$ letters acts on the $d$-fold
 product $X^d$ in the natural manner.
 }
 %%(End of Definition 2.6)%%%
\end{dfn}
%%%%%%%%%%%%%%%%%%%%%

%%(Remark 2.7)%%
\begin{rmk}
{\rm
An element $\eta\in \SP^d(X)$ may be identified with a formal linear
combination
$\eta =\sum_{k=1}^sn_kx_k$, where
$x_1,\cdots ,x_s$ are distinct points in $X$
and $n_1,\cdots ,n_s$ are positive integers such that
$\sum_{k=1}^sn_k=d$.
\qed
}
%%%%%%%%%%%%
\end{rmk}
%%%(End of Remark 2.7)%%%

%%(Definition 2.8)%%%
\begin{dfn}
%%%%%%%%%%%%%%%%%%%%%
{\rm
(i)
When $A\subset X$ is a closed subspace, define the equivalent relation
\lq\lq$\sim$\rq\rq \ on $\SP^d(X)$ by
%%(2.6)%%
\begin{equation}
%%%%%%%%
\eta_1\sim \eta_2\quad
\mbox{if }\quad
\eta_1 \cap (X\setminus A)=\eta_2 \cap (X\setminus A)
\quad
\mbox{for }\eta_1,\eta_2\in \SP^d(X).
\end{equation}
%%%
Define the space $\SP^d(X,A)$ by the quotient space
%%(2.7)%%
\begin{equation}
%%%%
\SP^d(X,A)=\SP^d(X)/\sim .
\end{equation}
%%%%%%%%
Note that the points in $A$ are ignored in $\SP^d(X,A)$.
If $A\not=\emptyset$, we have the natural inclusion 
$\SP^d(X,A)\subset \SP^{d+1}(X,A)$
by adding a point in $A$, and one can define the space
$\SP^{\infty}(X,A)$ by the union
%%(2.8)%%
\begin{equation}
\SP^{\infty}(X,A)=\bigcup_{d=1}^{\infty}\SP^d(X,A).
\end{equation}
%%(ii)%%%%%
\par
%%%%%%%%%
(ii)
For each $D=(d_1,\cdots ,d_r)\in (\Z_{\geq 1})^r$, let
$E_D^{\Sigma}(X)$ denote the space 
%%(2.9)%%
\begin{equation}
E_D^{\Sigma}(X)=
\{(\xi_1,\cdots ,\xi_r)
\in
\prod_{i=1}^r\SP^{d_i}(X)
:
\bigcap_{\i\in\sigma}\xi_i=\emptyset
\mbox{ for any }\sigma\in
I(\mathcal{K}_{\Sigma})\}.
\end{equation}
%%%%%
If $A\subset X$ is a closed subspace and $A\not=\emptyset$, define
the equivalence relation \lq\lq$\sim$\rq\rq \ on
the space $E^{\Sigma}_D(X)$ by
$$
(\xi_1,\cdots ,\xi_r)\sim
(\eta_1,\cdots ,\eta_r)
\quad
\mbox{if }\quad
\xi_i \cap (X\setminus A)=\eta_i \cap (X\setminus A)
\quad
\mbox{for each }1\leq j\leq r.
$$
Define the space
$E^{\Sigma}_D(X,A)$ by the quotient space
%%(2.10)%%
\begin{equation}
E^{\Sigma}_D(X,A)=E^{\Sigma}_D(X)/\sim .
\end{equation}
%%%%
Then
by adding points in $A$, we have the natural inclusion
$$
E^{\Sigma}_{D(k)}(X,A)\subset E^{\Sigma}_{D(k+1)}(X,A)
$$ for
$D(k)=(d_1+k,\cdots ,d_r+k)$ if $k\geq 0$.
So one can also define the space
$E^{\Sigma}_{}(X,A)$ by the union} 
%%(2.11)%%
\begin{equation}
%%%
E^{\Sigma}_{}(X,A)=\bigcup_{k\geq 0}
E^{\Sigma}_{D(k)}(X,A).
\end{equation}
%%%%%%%%%%%
%%%%
\end{dfn}
%%(End of Definition 2.8)%%%
%%
%%(Remark 2.9)%%
\begin{rmk}
%%%%
{\rm
(i)
It is easy to see that the space $E^{\Sigma}(X,A)$ does not depend on the choice of
the $r$-tuple $D=(d_1,\cdots ,d_r)$ and that the following equality holds:
%%()%%%
\begin{equation*}
%%%%%%
E^{\Sigma}_{}(X,A)=\{(\xi_1,\cdots ,\xi_r)\in \SP^{\infty}(X,A)^r:
\cap_{i\in\sigma}\xi_i=\emptyset
\mbox{ for any }\sigma\in
I(\mathcal{K}_{\Sigma})\}.
\end{equation*}
%%%
\par (ii)
Note that there is a natural homeomorphism $\P^d(\C)\cong\SP^d(\C)$
by identifying
$\prod_{k=1}^s(z-\alpha_k)^{n_k}\mapsto \sum_{k=1}^sn_k\alpha_k$,
where
$\{\alpha_k\}_{k=1}^s$ are distinct points in $\C$ and 
$\{n_k\}_{k=1}^s$ are positive integers such that
$\sum_{k=1}^sn_k=d$.
\par
(iii)
If the condition
(\ref{equ: homogenous}.2) is satisfied, by using the above identification
we have a natural homeomorphism
%%%(2.12)%%%
\begin{equation}\label{eq: identify}
%%%%%%%%%
\Hol_D^*(S^2,\XS)\cong E^{\Sigma}_D(\C).
\qed
\end{equation}
%%%%%%
}
\end{rmk}
%%%(End of Remark 2.9)%%

%%%(Definition 2.10)%%%
\begin{dfn}\label{dfn: stabilization etc}
%%%%%%%%
{\rm
Let $\textbf{\textit{a}}=(a_1,\cdots ,a_r)\in (\Z_{\geq 1})^r$ be any fixed
$r$-tuple of positive integers such that
%%(2.13)%%%
\begin{equation}\label{condition: a}
%%%%%%
\sum_{k=1}^ra_r\textbf{\textit{n}}_k=\textbf{0}.
\end{equation}
%%%%%%%
For each $E=(e_1,\cdots ,e_r)\in (\Z_{\geq 1})^r$, let
$U_E\subset \C$ denote te subset
$U_E=\{w\in\C:\mbox{ Re}(w)<e_1+\cdots +e_r\}$, and
choose any $r$  points
$\{x_j\}_{j=1}^r\subset U_{D+\textit{\textbf{a}}}\setminus
U_D$
such that $x_i\not= x_j$ if $i\not= j$.
For any such a choice, we define the stabilization map
$s_{D,\Sigma}:E_D^{\Sigma}(U_D)\to 
E_{D+\textbf{\textit{a}}}^{\Sigma}(U_{D+\textbf{\textit{a}}})$
by
%%%(2.14)%%
\begin{equation}
\begin{CD}
s_{D,\Sigma}:
E_D^{\Sigma}(U_D) @>>> E_{D+\textbf{\textit{a}}}^{\Sigma}(U_{D+\textbf{\textit{a}}})
\\
\ \ (\xi_1,\cdots ,\xi_r)
@>>>
(\xi_1+a_1x_1,\cdots ,\xi_r+a_rx_r)
\end{CD}
\end{equation}
%%%%
We also have the stabilization map
%%(2.15)%%
\begin{equation}\label{eq: sD}
%%%%%%%%
s_D:\Hol_d^*(S^2,\XS)\to \Hol_{D+\textbf{\textit{a}}}^*(S^2,\XS)
\end{equation}
%%%%%%%%
defined as the composite of maps
%%(2.16)%%
\begin{equation}
\Hol_D^*(S^2,\XS)\cong
E_D^{\Sigma}(U_D)\stackrel{s_{D,\Sigma}}{\longrightarrow} 
E_{D+\textbf{\textit{a}}}^{\Sigma}(U_{D+\textbf{\textit{a}}})
\cong
\Hol^*_{D+\textbf{\textit{a}}}(S^2,\XS).
\end{equation}
%%%%%
Note that, although the map $s_D$ depends on the choice of the points
$\{x_k\}_{k=1}^r$,  its homotopy class does not.
%%%%%%
}
%%%%%%
\end{dfn}
%%(End of Definition 2.10)%%%%%

%%%%%%%%%%%%%%%%%%%%%
Now we are ready to define the scanning map.
%%
%%(Definition 2.11)%%
\begin{dfn}
%%%%%%%%%%%%%%%%
{\rm
Let $\epsilon_0>0$ be any fixed sufficiently small number and let
$U=\{w\in \C:\vert w\vert <1\}$.
For each $w\in \C$, let $U_w=\{x\in \C:\vert x-w\vert <\epsilon_0\}$.
%\par
Then
for an element $\eta =(\eta_1,\cdots ,\eta_r)\in E^{\Sigma}_D(\C),$
define a map
%%()%%
$
sc_D(\eta):\C \to 
E^{\Sigma}_{}(\overline{U},\partial \overline{U})
$
by 
$$
w\mapsto
\eta \cap \overline{U}_w=
(\eta_1\cap \overline{U}_w,\cdots ,\eta_r\cap \overline{U}_w)
\in E^{\Sigma}(\overline{U}_w,\partial \overline{U}_w)
\cong
E_{}^{\Sigma}(\overline{U},\partial \overline{U})
$$
for $w\in \C$, where we identify
$(\overline{U}_w,\partial \overline{U}_w)$
with $(\overline{U},\partial \overline{U})$
in the canonical way.
Since $\dis \lim_{w\to\infty}sc(\eta)(w)=(\emptyset,\cdots ,\emptyset),$
it naturally extends to a map
%%%(2.17)%%
\begin{equation}\label{equ: sc}
%%%%%%%%%
sc(\eta):S^2=\C\cup\infty \to E^{\Sigma}_{}(\overline{U},\partial \overline{U})
\end{equation}
%%%%%%%%%%
by taking $sc(\eta)(\infty)=(\emptyset ,\cdots ,\emptyset).$
%%
%%%%%
Now we choose the point $\infty$ and the empty configuration
$(\emptyset,\cdots ,\emptyset)$ as the base-points of 
$S^2=\C \cup \infty$ and 
$E^{\Sigma}_{}(\overline{U},\partial \overline{U})$, respectively. 
Then the map $sc(\eta)$ is a base-point preserving map, and
we obtain a map $sc:E^{\Sigma}_D(\C)\to
\Omega^2E^{\Sigma}_{}(\overline{U},\partial \overline{U})$.
However,
since $E^{\Sigma}_D(\C)$ is connected, the image of the map $sc$ is contained some path-component of $\Omega^2E^{\Sigma}(\overline{U},\partial \overline{U})$,
which we denote by
$\Omega^2_DE^{\Sigma}_{}(\overline{U},\partial \overline{U}).$ 
Thus we have the map
%%%(2.18)%%
\begin{equation}\label{equ: scanning}
sc_D:E^{\Sigma}_D(\C)\to
\Omega^2_D E^{\Sigma}_{}(\overline{U},\partial \overline{U}).
\end{equation}
%%%
%%
If the condition (\ref{equ: homogenous}.2) is satisfied,
we can identify $\Hol_D^*(S^2,\XS)=E^{\Sigma}_D(\C)$ and 
we obtain 
the map
%%(2.19)%%%%%
\begin{equation}\label{equ: scanning map}
%%%%%%%%%%%
sc_D:\Hol_D^*(S^2,\XS)
\to
\Omega^2_D E^{\Sigma}(\overline{U},\partial \overline{U}).
\end{equation}
%%%
\par We refer to this map (and others defined by the same method) as 
\lq\lq\ the scanning map\rq\rq. 
%%%
It is easy to see that there is a commutative diagram
%%(2.20)%%
\begin{equation}\label{CD1}
%%%%%%%%
\begin{CD}
\Hol_D^*(S^2,\XS) @>sc_D>>\Omega^2_D E^{\Sigma}(\overline{U},\partial \overline{U})
\\
@V{s_D}VV   @VV{\simeq}V
\\
\Hol_{D+\textbf{\textit{a}}}^*(S^2,\XS) @>sc_{D+\textbf{\textit{a}}}>>
\Omega^2_{D+\textbf{\textit{a}}} E^{\Sigma}(\overline{U},\partial \overline{U})
\end{CD}
\end{equation}
%%%
Let $\dis \Hol_{D+\infty}^*(S^2,\XS)
=\lim_{k\to\infty}\Hol_{D+k\textbf{\textit{a}}}^*(S^2,\XS)$
denote the colimit constructed from the maps
$s_{D+k\textbf{\textit{a}}}$.
Then by using (\ref{CD1}) we obtained 
{\it the stabilized scanning map}
%%(2.21)%%
\begin{equation}
S:
\Hol_{D+\infty}^*(S^2,\XS)=
\lim_{k\to\infty}\Hol_{D+k\textbf{\textit{a}}}^*(S^2,\XS)
\to
\Omega^2_0 E^{\Sigma}(\overline{U},\partial \overline{U}),
\end{equation}
%%%
where we set $S=\lim_{k\to\infty}sc_{D+k\textbf{\textit{a}}}$ and
$\Omega^2_0 X$
denotes the path component of
$\Omega^2 X$
which contains the constant map.
}
\end{dfn}
%%%%(End of Definition 2.11)%%%

%%%%(Theorem 2.12: scanning map)%%
\begin{thm}\label{thm: scanning map}
%%%%%%%%%%%%%%%%%%%%%%%%%%%%%%%
The stabilized scanning map
$$
S:\Hol_{D+\infty}^*(S^2,\XS)=
\lim_{k\to\infty}\Hol_{D+k\textbf{\textit{a}}}^*(S^2,\XS)
\stackrel{\simeq}{\longrightarrow}
\Omega^2_0 E^{\Sigma}(\overline{U},\partial \overline{U})
$$
is a homotopy equivalence.
\qed
\end{thm}
%%%%%%%%%%%%%%%%%%
\begin{proof}
%%%%%
The assertion can be proved
by using Segal's scanning  method given in
\cite[Prop. 4.4]{Gu2} (cf. \cite{Gu1}) and \cite{GKY2}.
%Alternatively, one can use 
%\cite[Theorem 5.8]{Ka2}.
%%%
\end{proof}
%%%%%(End of proof of Theorem 2.12)%%%

%%(End of SECTION 2)%%%%%%
%%
%%
%%
%%
%%
%%
%%(SECTION 3)%%%
\section{Segal-type fibration sequences}
\label{section 3}
%%%%%%%%%%%%%%%%

%%%%%%%%%%%%%%%

In this section we recall the basic fact concerning the topology of
polyhedral products and determine the homotopy type
of the space $E^{\Sigma}(\overline{U},\partial \overline{U})$.
We also consider relation between the topologies of the spaces
$E^{\Sigma}(\overline{U},\partial \overline{U})$ and $\XS$,
and construct some Segal-type fibration sequences
(Proposition \ref{prp: Segal fibration}).

%%(Definition 3.1)%%
\begin{dfn}\label{dfn: veeK}
%%%%
{\rm
Let $(X,*)$ be a based space and let $I$ be 
a collection of some subsets of $[r]=\{1,2,\cdots ,r\}$.
Then let $\vee^IX\subset X^r$ denote the subspace consisting of all
$r$-tuples $(x_1,\cdots ,x_r)\in X^r$ such that,
for each $\sigma\in I$, there is some $j\in \sigma$
such that $x_j=*$.
}
\end{dfn}
%%%%%%%%%%%%%(End of Definition 3.1)%%
%%

%%(Lemma 3.2)%%
\begin{lmm}[cf. \cite{KY6}, Lemma 6.3]\label{lmm: veeK}
%%%%%%%%%%%%%
If $K$ is a simplicial complex on the index set $[r]$ and
$(X,*)$ is a based space, then
$\mathcal{Z}_K(X,*)=\vee^{I(K)}X$.
\end{lmm}
%%%%(Proof of Lemma 3.2)%%%%
\begin{proof}
%%%%%%%%%%%%%%%%%%%%%%%%%%%%
Suppose that $(x_1,\cdots,x_r)\in \mathcal{Z}_K(X,*)$.
Then there exists some $\sigma\in K$ such that 
$(x_1,\cdots ,x_r)\in (X,*)^{\sigma}$, that is, 
$x_j=*$ if $j\notin\sigma$.
%%%
Note that $\tau\not\subset \sigma$ for any
$\tau \in I(K)$; for  if $\tau\subset \sigma$ for some $\tau\in I(K)$,
then, by the definition of simplicial complex,  
$\tau\in K$, which is a contradiction.
%%%
Hence, for each $\tau\in I(K)$, there is some $j\in \tau\setminus\sigma$,
and so that $x_j=*$.
Thus $(x_1,\cdots ,x_r)\in \vee^{I(K)}X$, and we see that
$\mathcal{Z}_K(X,*)\subset \vee^KX$.
\par
Conversely, suppose that $(x_1,\cdots ,x_r)\in \vee^{I(K)}X$.
Let $\sigma =\{i\in [r]:x_i\not= *\}$.
Then  $\sigma\in K$; for if $\sigma\in I(K)$,  there is some $j\in \sigma$ such that
$x_j=*$. 
Since $j\in \sigma$ we must have $x_j\not= *$,
which is a contradiction.
Also, form the definition of $\sigma$,
$x_j=*$ if $j\notin\sigma$.
Hence, $(x_1,\cdots ,x_r)\in (X,*)^{\sigma}\subset \mathcal{Z}_K(X,*)$ and
 $\vee^KX\subset \mathcal{Z}_K(X,*)$.
\end{proof}
%%%%(End of proof of Lemma 3.2)%%%

%%%(Lemma 3.3)%%
\begin{lmm}\label{lmm: E-infty}
%%%%%%%%%%%%%%%%
There is a homotopy equivalence
$r_{\Sigma}:E^{\Sigma}_{}(\overline{U},\partial \overline{U})
\stackrel{\simeq}{\longrightarrow}
DJ(\KS).$
\end{lmm} 
%%%%%%%%%%%%%%%%
\begin{proof}
%%%%%%%%%%%%%%%
The proof is analogous to that of
\cite[Prop. 3.1]{Se}
and that of \cite[Lemma 7.10]{Ka1}.
Note that
$E^{\Sigma}_{}(\overline{U},\partial \overline{U})$ is homeomorphic to the space
$$
E^{\Sigma}_{}(S^2,\infty)
=\{(\eta_1,\cdots ,\eta_r)\in \SP^{\infty}(S^2,\infty)^r:
\cap_{i\in\sigma}\eta_i=\emptyset
\mbox{ for any }\sigma\in I(\KS)\}.
$$
For each $\epsilon >0$, let $E_{\epsilon}^{\Sigma}$ denote the open subset
of $E^{\Sigma}_{}(S^2,\infty)$
consisting of all $r$-tuples
$(\xi_1,\cdots ,\xi_r)\in
E^{\Sigma}_{}(S^2,\infty)$
such that, for any $\sigma \in I(\mathcal{K}_{\Sigma})$ there
exists some $i\in\sigma$  satisfying the condition $\xi_i\cap
\overline{U(\epsilon)}=\emptyset$.
%%%
For each $\epsilon >0$, radial expansion defines a deformation retraction
$r_{\epsilon}:E^{\Sigma}_{\epsilon}\stackrel{\simeq}{\rightarrow}
\vee^{I(\mathcal{K}_{\Sigma})}\SP^{\infty}(S^2,\infty)$
(in this case, if $\xi_i\cap \overline{U(\epsilon)}=\emptyset$
and $i\in \sigma$ (for any $\sigma\in I(K)$),
then $\xi_i$ gets retracted to $\infty$).
Since $E^{\Sigma}_{}(S^2,\infty)=
\bigcup_{\epsilon>0}E^{\Sigma}_{\epsilon}$,
there is a deformation retraction
$E^{\Sigma}((S^2,\infty)
\stackrel{\simeq}{\rightarrow}
\vee^{I(\mathcal{K}_{\Sigma})}\SP^{\infty}(S^2,\infty).$
Since there is a homeomorphism $\SP^{\infty}(S^2,\infty)\cong
\CP^{\infty}$,  
the assertion follows from Lemma \ref{lmm: veeK}.
%%%%%%%%%
\end{proof}
%%%%(End of Proof of Lemma 3.3)%%

%%(Proposition 3.4)%%
\begin{prp}\label{prp: Segal fibration}
%%%%%%%%%%
If $\{\textit{\textbf{n}}_k\}_{k=1}^r$ spans $\R^n$ and
$\XS$ is non-singular, there is a fibration sequence
(up to homotopy)
%%(3.1)%%
\begin{equation}\label{eq: Segal fibration}
%%%%%%%%%
\T^n_{\C} \stackrel{}{\longrightarrow} \XS \stackrel{p_{\Sigma}}{\longrightarrow} 
DJ(\mathcal{K}_{\Sigma}).
\end{equation}
%%%
\end{prp}
%%%
\begin{proof}
%%%%%%
Let us write $(K,G,G_1)=(\mathcal{K}_{\Sigma},\T^n_{\C},\T^r_{\C})$
and 
we identify %make the identification
$\XS =\mathcal{Z}_K(\C,\C^*)/\GS$.
Since $\GS\cong \T^{r-n}_{\C}$,
there is a fibration sequence
$\GS \stackrel{i}{\longrightarrow} \T^r_{\C}=G_1 \to \T^n_{\C}=G$, where
$i$ denotes the inclusion.
Note that $G=\T^n_{\C}$ acts naturally on $\XS$ by the definition of
toric variety and that we have a fibration sequence
$G\to \XS \to EG\times_G\XS.$
Moreover, since the group 
$G_1=\T^r_{\C}$ also acts on $\mathcal{Z}_{K}(\C,\C^*)$
and we also obtain a fibration sequence
$G_1 \to \mathcal{Z}_{K}(\C,\C^*)\stackrel{}{\longrightarrow}
EG_1\times_{G_1}\mathcal{Z}_{K}(\C,\C^*).$
Hence, by using \cite[Lemma 2.1]{CMN}
we have the following homotopy commutative diagram
$$
\begin{CD}
\GS @>i>\subset> G_1=\T^r_{\C} @>\pi>> G=\T^n_{\C}
\\
\Vert @.  @VVV @VVV
\\
\GS @>>> \mathcal{Z}_K(\C,\C^*) @>q_{\Sigma}>> \XS =\mathcal{Z}_K(\C,\C^*)/\GS
\\
@VVV @VVV @V{p_{\Sigma}}VV
\\
* @>>> EG_1\times_{G_1}\mathcal{Z}_K(\C,\C^*)
@>E\pi\times q_{\Sigma}>>
EG\times_G \XS
\end{CD}
$$
where 
all horizontal and vertical sequences are fibration sequences
(up to homotopy equivalence).
Thus, the map 
$E\pi\times q_{\Sigma}:
EG_1\times_{G_1}\mathcal{Z}_K(\C,\C^*)
\stackrel{\simeq}{\longrightarrow}
EG\times_G \XS$ 
is a homotopy equivalence.
However, it follows from \cite[Theorem 6.29]{BP}
that there is a homotopy equivalence
$EG_1\times_{G_1}\mathcal{Z}_K(\C,\C^*)\simeq DJ(K)$.
Thus, there is a homotopy equivalence
$EG\times_G\XS\simeq DJ(K)$ and the assertion follows.
%%%
\end{proof}
%%%(End of proof fo Proposition 3.4)%%

%%(Example 3.5)%%
\begin{example}\label{example: Segal fibration}
%%%%%%%%%%%%%%%%%
{\rm
For a based space $(X,*)$, let
$W_{n+1}(X)$ denote the fat wedge of $X$ defined by
$
W_{n+1}(X)=\{
(x_0,\cdots ,x_n)\in X^{n+1}:x_i=*
\mbox{ for some }0\leq i\leq n\},
$
and
let $\textit{\textbf{e}}_k$ $(1\leq k\leq n$) be the standard basis of $\R^n$ given by
$$
\textit{\textbf{e}}_1=(1,0,0,\cdots ,0),
\textit{\textbf{e}}_2=(0,1,0,\cdots ,0),
\cdots
%,\textit{\textbf{e}}_{n-1}=(0,\cdots ,0,1,0),
,\textit{\textbf{e}}_n=(0,0,\cdots ,0,1).
$$
It is known that the fan $\Sigma$ of $\CP^n$ is given by
$$
\Sigma =
\{\mbox{Cone}(S):
S\subsetneqq \{\textit{\textbf{e}}_k\}_{k=0}^n\},
\mbox{ where we set }\textit{\textbf{e}}_0=-\sum_{k=1}^n\textit{\textbf{e}}_k.\footnote{%
%%%(FootNote 3)%%%%%%%
For $S=\emptyset$, we set $\mbox{Cone}(\emptyset )=\{{\bf 0}\}$.}
%%%(End of FootNOte)%%%%%
%%
%%
$$
%where we set $\textit{\textbf{e}}_0=-\sum_{k=1}^n\textit{\textbf{e}}_k$.
Hence, if $\XS =\CP^n$, we see that
$DJ(\KS)=W_{n+1}(\CP^{\infty})$ 
and 
the above fibration sequence (\ref{eq: Segal fibration}) 
coincides with the following fibration sequence
%%%
%%%(2.14)%%
\begin{equation}\label{eq: CP sequence}
%%%%
\T^n_{\C} \stackrel{}{\longrightarrow} \CP^n \stackrel{p_{\Sigma}}{\longrightarrow}
W_{n+1}(\CP^{n+1}).
\end{equation}
%%%
constructed by G. Segal in
\cite[\S 2]{Se},
\qed
}
%%%%
\end{example}
%%(End of Example 3.5)%%

%%(End of SECTION 3)%%%%%%%%%

%%%%(SECTION 4)%%%%
\section{The stable result}\label{section: stability}
%%%%%%%%%%%%%%%%%%%

%%(Definition 4.1)%%%
\begin{dfn}
%%%%
{\rm
Let $\textit{\textbf{a}}\in (\Z_{\geq 1})^r$ be an $r$-tuple of positive integers which satisfies
the condition (\ref{condition: a}).
Then it is easy to see that the following diagram is homotopy commutative:
$$
\begin{CD}
\Hol_D^*(S^2,\XS) @>i_D>> \Omega^2_D\XS @>>\simeq> \Omega_0^2\XS 
\\
@V{s_D}VV @V{\simeq}VV \Vert @.%@V{\simeq}VV
\\
\Hol_{D+\textit{\textbf{a}}}^*(S^2,\XS) 
@>i_{D+\textit{\textbf{a}}}>> 
\Omega^2_{D+\textit{\textbf{a}}}\XS  @>>\simeq> \Omega^2_0\XS
\end{CD}
$$
%%%
Hence we can stabilize the inclusion maps
%%
%%()%%%%%%%%%
\begin{equation*}\label{eq: inclusion stab}
%%%%%%%%%%%%%%%%
i_{D+\infty}=\lim_{k\to\infty}i_{D+k\textit{\textbf{a}}}:
\Hol^*_{D+\infty}(S^2,\XS)
=
\lim_{k\to\infty}\Hol_{D+k\textbf{\textit{a}}}^*(S^2,\XS) \to \Omega^2_0\XS .
\end{equation*}
%%%
}
\end{dfn}
%%(End of definition 4.1)%%%

In this section, we shall prove the following result.
%%%(Theorem 4.2)%%
\begin{thm}\label{thm: II}
%%%
The map
$i_{D+\infty}:
\Hol^*_{D+\infty}(S^2,\XS)
\stackrel{\simeq}{\longrightarrow}
\Omega^2_0\XS 
$
is a homotopy equivalence.
\end{thm}
%%%%

%%%(Definition 4.3)%%%
\begin{dfn}
%%%
{\rm
(i)
Let $X\subset \C$ be an open set and let $F(X)$ denote the space
of $r$-tuples $(f_1(z),\cdots ,f_r(z))$
of (not necessarily monic) polynomials
satisfying the following two conditions:
\begin{enumerate}
\item[(i-1)]
$\sum_{k=1}^r\deg (f_k)\textit{\textbf{n}}_k=\mathbf{0}$.
\item[(i-2)]
For each $\sigma =\{i_1,\cdots ,i_s\}\in I(\mathcal{K}_{\Sigma})$,
the polynomials $f_{i_1}(z),\cdots ,f_{i_s}(z)$ have no common root in $X$.
\end{enumerate}
%%%%
An element 
$(f_1(z),\cdots ,f_r(z))\in F(X)$,
defines a map $X\to U(\mathcal{K}_{\Sigma})$ and represents a map
$X\to \XS$.
\par
(ii) Let $U=\{w\in \C:\vert w\vert <1\}$ and $ev_0:F(U)\to U(\mathcal{K}_{\Sigma})$ denote the map given by evaluation at $0$, i.e.
$ev_0(f_1,\cdots ,f_r)=(f_1(0),\cdots ,f_r(0))$.
\par
(iii)
Let $\tilde{F}(U)\subset F(U)$ denote the subspace
%consisting 
of all
$(f_1(z),\cdots ,f_r(z))\in F(X)$ such that no $f_i(z)$ is identically
zero, and
$ev:\tilde{F}(U)\to U(\mathcal{K}_{\Sigma})$
the map given by the restriction
$ev=ev_0\vert \tilde{F}(U)$.
%%%
}
%%%%%%%%
\end{dfn}
%%(End of Definition 4.3)%%%%%

%%(Lemma 4.4)%%
\begin{lmm}\label{lmm: ev}
%%%
$ev:\tilde{F}(U)\stackrel{\simeq}{\longrightarrow}
U(\mathcal{K}_{\Sigma})$
is a homotopy equivalence.
%%%
\end{lmm}
%%%
\begin{proof}
%%%%%
Let $i_0:U(\mathcal{K}_{\Sigma})\to F(U)$ be the inclusion map given by viewing constants as polynomials. 
Clearly $ev_0\circ i_0=\mbox{id}$.
Let $f: F(U)\times [0,1]\to
F(U)$ be
the homotopy given by
$f((f_1,\cdots ,f_t),t)=(f_{1,t}(z),\cdots ,f_{r,t}(z))$,
where
$f_{i,t}(z)=f_i(tz)$.
This gives a homotopy between
$i_0\circ ev_0$ and the identity map, and this proves that $ev_0$ is a homotopy equivalence.
Since $F(U)$ is an infinite dimensional manifold and
$\tilde{F}(U)$ is a subspace of $F(U)$ of infinite codimension,
the inclusion
$\tilde{F}(U)\to F(U)$ is a homotopy equivalence.
Hence $ev$ is also a homotopy equivalence.
%%%
\end{proof}
%%%%%%%%%(End of proof of Lemma 4.4)%%%%

%%%(Proof of Theorem 4.2)%%
\begin{proof}[Proof of Theorem \ref{thm: II}]
%%%%%%%%%%%%%%%%%%%%%%%%%%
Let $U=\{w\in\C: \vert w\vert <1\}$ as before and
note that the group $\T^r_{\C}$ acts freely on the space
$\tilde{F}(X)$
by coordinate multiplication for $X=U$ or $\C$.
Let $\tilde{F}(X)/\T^r_{\C}$ denote the corresponding orbit space.
Let $u:\tilde{F}(U)\to E^{\Sigma}(\overline{U}, \partial \overline{U})$
denote the natural map which
assigns to an $r$-tuple  $[f_1(z),\cdots ,f_r(z)]\in \tilde{F}(U)$ of polynomials
the $r$-tuple of their roots which lie in $U$.
It is not difficult to see that the map $u$ is a homotopy equivalence.
Let  $scan: \tilde{F}(\C)\to \Map (\C,\tilde{F}(U))$ denote the map
given by $scan (f_1(z),\cdots ,f_r(z))(w)=(f_1(z+w),\cdots ,f_r(z+w))$
for $w\in\C$ and consider the diagram
$$
\begin{CD}
\tilde{F}(U) @>ev>\simeq>U(\mathcal{K}_{\Sigma})
\\
@V{p}VV @.
\\
\tilde{F}(U)/\T^r_{\C} @>u>\simeq> 
E^{\Sigma}(\overline{U},\partial \overline{U})
\end{CD}
$$
where
$p:\tilde{F}(U)\to \tilde{F}(U)/\T^r_{\C}$
denotes the natural projection map.
Note that $p$ is a $\T^r_{\C}$-principal bundle projection. 
%%%
%%%
Now consider the diagram below
$$
\begin{CD}
\tilde{F}(\C) @>scan>> \Map (\C, \tilde{F}(U)) 
@>ev_{\#}>\simeq> \Map (\C,U(\mathcal{K}_{\Sigma}))
\\
@V{p}VV @V{p_{\#}}VV @.
\\
\tilde{F}(\C)/\T^r_{\C} @>scan>>
\Map (\C, \tilde{F}(U)/\T^r_{\C}) 
@>u_{\#}>{\simeq}>
\Map (\C, E^{\Sigma}(\overline{U},\partial \overline{U}))
\end{CD}
$$
induced from the above diagram.
Observe that $\Map (\C,\cdot )$ can be replaced by $\Map (S^2,\cdot)$
by extending  from $\C$ to $S^2=\C\cup \infty$
(as base point preserving maps).
%%%
%%%
Thus by setting
$$
\begin{cases}
j_D:\Hol_D^*(S^2,\XS) \stackrel{\subset}{\longrightarrow} \tilde{F}(\C) 
\stackrel{scan}{\longrightarrow}
\Map_D^*(S^2,\tilde{F}(U))
=\Omega^2_D\tilde{F}(U)
\\
j_D^{\p}:E^{\Sigma}_D(\C)
\stackrel{\subset}{\longrightarrow}
\tilde{F}(\C)/\T^r_{\C} \stackrel{scan}{\longrightarrow}
\Map_D^*(S^2,\tilde{F}(U)/\T^r_{\C})
=\Omega^2_D(\tilde{F}(U)/\T^t_{\C})
\end{cases}
$$
we obtain the following commutative diagram, where
the suffix $D$ denotes the appropriate path component:
%%%%(The main diagram)%%
$$
\begin{CD}
%%%%%%
\Hol_D^*(S^2,\XS) 
@>j_D>> \Omega^2_D\tilde{F}(U) 
@>\Omega^2 ev>\simeq> \Omega^2_DU(\mathcal{K}_{\Sigma})
@>\Omega^2q_{\Sigma}>\simeq> \Omega^2_D \XS
\\
@V{\cong}VV 
 @V{\Omega^2p}V{\simeq}V @. @.
\\
E^{\Sigma}_D(\C)
@>j_D^{\p}>>
\Omega^2_D (\tilde{F}(U)/\T^r_{\C}) 
@>\Omega^2 u>{\simeq}>
\Omega^2_DE^{\Sigma}(\overline{U},\partial \overline{U})
@.
\end{CD}
$$
Note that the maps
$\Omega^2q_{\Sigma}$, $ev$, $\Omega^2p$ and $u$ are homotopy equivalences.
Moreover, from the definitions of the maps, one can see that the following two equalities hold (up to homotopy equivalence):
%%(4.2)%%
\begin{equation}
\Omega^2q_{\Sigma}\circ \Omega^2ev\circ  j_D=i_D,
\quad 
\Omega^2u\circ j_D^{\p}=sc_D.
\end{equation}
%%%
Hence, the maps $i_D$ and $sc_D$ are homotopic up to homotopy equivalences.
Thus, if we replace $D$ by $D+k\textit{\textbf{a}}$
and let $k\to\infty$ then,
by using Theorem \ref{thm: scanning map},
we see that the map $i_{D+\infty}$ is a homotopy equivalence.
%%%%%%%%%%%%%%%%%%%%%%%%
\end{proof}
%%(End of Proof of Theorem 4.2)%%%
%%%%%%%

%%'Remark 4.5)%%
\begin{rmk}
%%%
{\rm
It is not easy to show that 
the following diagram is homotopy commutative:
%%(4.3)%%
\begin{equation}\label{eq: diagram 4.3}
\begin{CD}
\tilde{F}(U) @>ev>\simeq> U(\mathcal{K}_{\Sigma}) @>q_{\Sigma}>> \XS
\\
@V{p}VV @. @V{p_{\Sigma}}VV
\\
\tilde{F}(U)/\T^r_{\C} @>u>{\simeq}> 
E^{\Sigma}(\overline{U},\partial \overline{U}) @>r_{\Sigma}>{\simeq}> 
DJ(\mathcal{K}_{\Sigma}) 
\end{CD}
\end{equation}
%%%%
Segal proved that the diagram (\ref{eq: diagram 4.3}) is homotopy
commutative for the case
$\XS =\CP^n$
(see \cite[Prop. 4.8]{Se}), and an analogous  method can probably be used to show that this diagram
is homotopy commutative for a general $\XS$. However, the argument seems tedious.
So we do not attempt to carry it out. 
Of course %(\ref{eq: diagram 4.3}) 
this result would give another proof of
Theorem \ref{thm: II}.
\qed
%%%%
}
%%%%
\end{rmk}
%%%%%%%(End of Remark 4.5)%%%%%

%%(SECTION 5)%%%
\section{Simplicial resolutions}\label{section: simplicial resolution}
%%%%%%%%%%%%%%%%

In this section, we summarize  the definitions of the non-degenerate simplicial resolution
and the associated truncated simplicial resolutions (\cite{Mo2},  \cite{Va}).
%%%%%
%%(Definition 5.1)%%
\begin{dfn}\label{def: def}
%%%%%%
{\rm
(i) For a finite set $\textbf{\textit{v}} =\{v_1,\cdots ,v_l\}\subset \R^N$,
let $\sigma (\textbf{\textit{v}})$ denote the convex hull spanned by 
$\textbf{\textit{v}}.$
%%%
%%%
Let $h:X\to Y$ be a surjective map such that
$h^{-1}(y)$ is a finite set for any $y\in Y$, and let
$i:X\to \R^N$ be an embedding.
Let  $\mathcal{X}^{\Delta}$  and $h^{\Delta}:{\mathcal{X}}^{\Delta}\to Y$ 
denote the space and the map
defined by
%%%
%%(5.1)%%%%%%%%
\begin{equation}
%%%%%%%%%%%%%%%
\mathcal{X}^{\Delta}=
\big\{(y,u)\in Y\times \R^N:
u\in \sigma (i(h^{-1}(y)))
\big\}\subset Y\times \R^N,
\ h^{\Delta}(y,u)=y.
\end{equation}
%%%%%%%%%%%%%%%%
The pair $(\mathcal{X}^{\Delta},h^{\Delta})$ is called
{\it the simplicial resolution of }$(h,i)$.
In particular, it %$(\mathcal{X}^{\Delta},h^{\Delta})$
is called {\it a non-degenerate simplicial resolution} if for each $y\in Y$
any $k$ points of $i(h^{-1}(y))$ span $(k-1)$-dimensional simplex of $\R^N$.
%%%%
\par
(ii)
For each $k\geq 0$, let $\mathcal{X}^{\Delta}_k\subset \mathcal{X}^{\Delta}$ be the subspace
given by 
%%(5.2)%%
\begin{equation}
%%%%%%%%
\mathcal{X}_k^{\Delta}=\big\{(y,u)\in \mathcal{X}^{\Delta}:
u \in\sigma (\textbf{\textit{v}}),
\textbf{\textit{v}}=\{v_1,\cdots ,v_l\}\subset i(h^{-1}(y)),l\leq k\big\}.
\end{equation}
%%%%
%%%%
We make the identification $X=\mathcal{X}^{\Delta}_1$ by identifying 
 $x\in X$ with the pair
$(h(x),i(x))\in \mathcal{X}^{\Delta}_1$,
and we note that  there is an increasing filtration
%%%(5.3)%%
\begin{equation}\label{equ: filtration}
%%%
\emptyset =
\mathcal{X}^{\Delta}_0\subset X=\mathcal{X}^{\Delta}_1\subset \mathcal{X}^{\Delta}_2\subset
\cdots \subset \mathcal{X}^{\Delta}_k\subset 
%\mathcal{X}^{\Delta}_{k+1}\subset
\cdots \subset \bigcup_{k= 0}^{\infty}\mathcal{X}^{\Delta}_k=\mathcal{X}^{\Delta}.
\end{equation}
%%%%%
Since the map $h^{\Delta}:\mathcal{X}^{\Delta}\stackrel{}{\rightarrow}Y$
 is a proper map, it extends to the map
 ${h}_+^{\Delta}:\mathcal{X}^{\Delta}_+\stackrel{}{\rightarrow}Y_+$
 between the one-point compactifications,
 where $X_+$ denotes the one-point compactification of a locally compact space $X$.
}
\end{dfn}
%%%%%(End of Definition 5.1)%%%%%%%%

%%%(Lemma 5.2)%%
\begin{lmm}[\cite{Va}, \cite{Va2} (cf. 
Lemma 3.3 in \cite{KY7})]\label{lemma: simp}
%%%%%%%%
Let $h:X\to Y$ be a surjective map such that
$h^{-1}(y)$ is a finite set for any $y\in Y,$ and let
$i:X\to \R^N$ be an embedding.
\par
%%(i)%%
$\I$
If $X$ and $Y$ are semi-algebraic spaces and the
two maps $h$, $i$ are semi-algebraic maps, then the map
${h}^{\Delta}_+:\mathcal{X}^{\Delta}_+\stackrel{\simeq}{\rightarrow}Y_+$
is a homotopy equivalence.
%%%
\par
$\II$
There is an embedding $j:X\to \R^M$ such that
the associated simplicial resolution
$(\tilde{\mathcal{X}}^{\Delta},\tilde{h}^{\Delta})$
of $(h,j)$ is non-degenerate.
%%%
%%%%
\par
$\III$
If there is an embedding $j:X\to \R^M$ such that the associated simplicial resolution
$(\tilde{\mathcal{X}}^{\Delta},\tilde{h}^{\Delta})$
of $(h,j)$ is non-degenerate,
the space $\tilde{\mathcal{X}}^{\Delta}$
is uniquely determined up to homeomorphism.
Moreover,
there is a filtration preserving homotopy equivalence
$q^{\Delta}:\tilde{\mathcal{X}}^{\Delta}\stackrel{\simeq}{\rightarrow}{\mathcal{X}}^{\Delta}$ such that $q^{\Delta}\vert X=\mbox{id}_X$.
\qed
%%%
\end{lmm}
%%%(End of Lemma 5.2)%%

%%%%%(Remark 5.3)%%%
\begin{rmk}\label{Remark: homotopy equ}
%%%%
{\rm
In this paper we only need the weaker assertion that the map ${h}_+^{\Delta}$ is a homology equivalence. 
One can easily prove this result
%that the map   $\overline{h}^{\Delta}$ is a homology equivalence 
by the same argument as used 
%in the proof of  Lemma 1 (page 90) 
in the second  revised edition of Vassiliev's book \cite[Proof of Lemma 1 (page 90)]{Va}.
\qed}
%%%(End of FootNote 1)%%
%%% 
%\qed
%%%%(ii)%%%
\end{rmk}

%%%%(Remark 5.4)%%%%%%%%%%%
%%%
\begin{rmk}[\cite{Va}, \cite{Va2}]\label{Remark: non-degenerate}
{\rm
Even for a  surjective map $h:X\to Y$ which is not finite to one,  
it is still possible to construct an associated non-degenerate simplicial resolution.
Recall that it is known that there exists a sequence of embeddings
$\{\tilde{i}_k:X\to \R^{N_k}\}_{k\geq 1}$ satisfying the following two conditions
for each $k\geq 1$ (\cite{Va}, \cite{Va2}).
%%%(Condition of non-degenerate simp. resolution)%%
\begin{enumerate}
%\item[(\ref{Remark: non-degenerate}$)_k$]
%%%%%%%%%%%%%%%%%%%
%\begin{enumerate}
%%(i)%%%
\item[(i)]
For any $y\in Y$,
any $t$ points of the set $\tilde{i}_k(h^{-1}(y))$ span $(t-1)$-dimensional affine subspace
of $\R^{N_k}$ if $t\leq 2k$.
%%(ii)%%%
\item[(ii)]
$N_k\leq N_{k+1}$ and if we identify $\R^{N_k}$ with a subspace of
$\R^{N_{k+1}}$, 
then $\tilde{i}_{k+1}=\hat{i}\circ \tilde{i}_k$,
where
$\hat{i}:\R^{N_k}\stackrel{\subset}{\rightarrow} \R^{N_{k+1}}$
denotes the inclusion.
%\end{enumerate}
\end{enumerate}
%%%%%%%%%%%%
In this situation, in fact,
a non-degenerate simplicial resolution may be
constructed by choosing a sequence of embeddings
$\{\tilde{i}_k:X\to \R^{N_k}\}_{k\geq 1}$ satisfying the above two conditions
for each $k\geq 1$.

%%%
Let
%%%()%%%%%
%begin{equation*}\label{2.1}
$\dis\mathcal{X}^{\Delta}_k=\big\{(y,u)\in Y\times \R^{N_k}:
u\in\sigma (\textbf{\textit{v}}),
\textbf{\textit{v}}
=\{v_1,\cdots ,v_l\}\subset \tilde{i}_k(h^{-1}(y)),l\leq k\big\}.$
%\end{equation*}
%%%%
%%
Then
by identifying naturally  ${\cal X}^{\Delta}_k$ with a subspace
of ${\cal X}_{k+1}^{\Delta}$,  define the non-degenerate simplicial
resolution ${\cal X}^{\Delta}$ of  $h$ as %the union  
$\dis {\cal X}^{\Delta}=\bigcup_{k\geq 1} {\cal X}^{\Delta}_k$.
%Non-degenerate simplicial resolutions have a long been used in algebraic geometry, and play  the central role in the work of Vassiliev  \cite{Va}.
}
\qed
\end{rmk}
%%%%%(End of Remark 5.4)%%%%%%%%
%%%

%%%(Definition 5.5)%%%
\begin{dfn}\label{def: 2.3}
{\rm
Let $h:X\to Y$ be a surjective semi-algebraic map between semi-algebraic spaces, 
$j:X\to \R^N$ be a semi-algebraic embedding, and let
$(\mathcal{X}^{\Delta},h^{\Delta}:\mathcal{X}^{\Delta}\to Y)$
denote the associated non-degenerate  simplicial resolution of $(h,j)$. 
%Under these conditions the map $h^{\Delta}$ is a homotopy equivalence as in Lemma \ref{lemma: simp}.
%%%
\par
%%%
Let $k$ be a fixed positive integer and let
$h_k:\mathcal{X}^{\Delta}_k\to Y$ be the map
defined by the restriction
$h_k:=h^{\Delta}\vert \mathcal{X}^{\Delta}_k$.
%\par
The fibers of the map $h_k$ are $(k-1)$-skeleton of the fibers of $h^{\Delta}$ and, in general,  always
fail to be simplices over the subspace
$Y_k=\{y\in Y:\mbox{card}(h^{-1}(y))>k\}.$
%$$
%Y_k=\{y\in Y:h^{-1}(y)\mbox{ consists of more than $k$ points}\}.
%$$
Let $Y(k)$ denote the closure of the subspace $Y_k$.
We modify the subspace $\mathcal{X}^{\Delta}_k$ so as to make  all
the fibers of $h_k$ contractible by adding to each fibre of $Y(k)$ a cone whose base
is this fibre.
We denote by $X^{\Delta}(k)$ this resulting space and by
$h^{\Delta}_k:X^{\Delta}(k)\to Y$ the natural extension of $h_k$.
\par
%%%%%%%%%%%%
Following  \cite{Mo3}, we call the map $h^{\Delta}_k:X^{\Delta}(k)\to Y$
{\it the truncated $($after the $k$-th term$)$  simplicial resolution} of $Y$.
Note that 
that there is a natural filtration
$$
 X^{\Delta}_0\subset
 X^{\Delta}_1\subset
%X^{\Delta}_2\subset
\cdots 
\subset X^{\Delta}_l\subset X^{\Delta}_{l+1}\subset \cdots
\subset  X^{\Delta}_k\subset X^{\Delta}_{k+1}
=X^{\Delta}_{k+2}
%=X^{\Delta}_{k+3}
=\cdots =X^{\Delta}(k),
$$
where $X^{\Delta}_0=\emptyset$,
$X^{\Delta}_l=\mathcal{X}^{\Delta}_l$ if $l\leq k$ and
$X^{\Delta}_l=X^{\Delta}(k)$ if $l>k$.
}
\end{dfn}
%%%%
%%(Remark 4)%%%
%\begin{rem}\label{Remark 4}
%Let $\varphi:A\to B$ be a simplicial map between simplicial complexes and
%let $C_{f}\varphi :C_{f}A\to B$ be its fibrewise cone construction.
%Because $C_f\varphi$ can be constructed in the category of simplicial complexes,
%it is a quasi-fibration and so it is a homotopy equivalence \cite{Hatch}.
%It is also know that
%for a semi-algebraic map $f:X\to Y$ between semi-algebraic spaces, there are
%semi-algebraic triangulations on $X$ and $Y$ such that 
%there is a semi-algebraic trivialization of the map $f$
%\cite[Theorems 9.2.1 and 9.3.2]{BCR}.
%\end{rem}
%%%%
%Then using these results, it is not difficult to prove the following.

%%
%%(Lemma 5.6)%%%%%%%%%
\begin{lmm}[\cite{Mo3}, cf. Remark 2.4 and Lemma 2.5 in \cite{KY4}]\label{Lemma: truncated}
%%%%%%%%%%%%%%%%%%%%%%%
%Let $h:X\to Y$ be a surjective semi-algebraic map between semi-algebraic spaces,
%and 
%let $j:X\to \R^N$ be a semi-algebraic embedding with the associated simplicial resolution
%$({\cal X}^{\Delta},h^{\Delta}:{\cal X}^{\Delta}\to Y)$.
%Then
Under the same assumptions and with the same notation as in Definition \ref{def: 2.3}, the map
$h^{\Delta}_k:X^{\Delta}(k)\stackrel{\simeq}{\longrightarrow} Y$ 
is a homotopy equivalence.
\qed
\end{lmm}
%%%%%%%%%%%

%%%
%%%(SECTION 6)%%%
\section{The Vassiliev spectral sequence.}
\label{section: spectral sequence}
%%%%%%%%%%%%%%%%%

In this section, we always assume that
$D=(d_1,\cdots ,d_r)\in (\Z_{\geq 1})^r$ and $\textit{\textbf{a}}=(a_1,\cdots ,a_r)\in (\Z_{\geq 1})^r$
are $r$-tuples of positive integers which satisfy the conditions
(\ref{equ: homogenous}.2) and (\ref{condition: a}).

From now on, we identify the space
 $\Hol_D^*(S^2,\XS)$ with the space consisting 
of all
$r$-tuples
$(f_1(z),\cdots ,f_{r}(z))\in \P^D$
of monic polynomials
such that
$f_{i_1}(z),\cdots ,f_{i_s}(z)$
have no common root for any
$\sigma =\{i_1,\cdots ,i_s\}\in I(\mathcal{K}_{\Sigma})$
as in Definition \ref{dfn: holomorphic}.
%%
%%%%%%%%%%%%%%%%%%%
First, we construct the Vassiliev spectral sequence.

%%(Definition 6.1)%%%%%%%
\begin{dfn}\label{Def: 3.1}
{\rm
%%()%%%%
%%%%%(i)%%%
(i)
Let $\Sigma_D$ denote {\it the discriminant} of $\Hol_D^*(S^2,\XS)$ in $\P^D$ 
given by the complement
%%%%
\begin{align*}
\Sigma_D&=
\P^D \setminus \Hol_d^*(S^2,\XS)
\\
&=
\{(f_1(z),\cdots ,f_{r}(z))\in \P^D :
(f_1(x),\cdots ,f_{r}(x))\in L(\Sigma)
\mbox{ for some }x\in \C\},
\end{align*}
where we set

%%%(6.1)%%
\begin{equation}
%%%
L(\Sigma) =
\bigcup_{\sigma\in I(\mathcal{K}_{\Sigma})}L_{\sigma}
=
~\bigcup_{\sigma\subset[r],\sigma\notin K_{\Sigma}}
L_{\sigma}.
\end{equation}
%%%%%%
\par
%%(ii)%%
(ii)
Let  $Z_D\subset \Sigma_D\times \C$
denote %the subspace
{\it the tautological normalization} of 
 $\Sigma_D$ consisting of 
 all pairs 
$(F,x)=((f_1(z),\ldots ,f_{r}(z)),
x)\in \Sigma_D\times\C$
satisfying the condition
$(f_1(x),\cdots ,f_{r}(x))\in L(\Sigma)$.
%%%
%\par
%%%%
Projection on the first factor  gives a surjective map
$\pi_D :Z_D\to \Sigma_D$.
%%%
%%%
}
\end{dfn}
%%%%%%
%%(End of Definition 6.1)%%%%

%%(Remark 6.2)%%
\begin{rmk}
%%%%%%%%%%%%%%%
{\rm
Let $\sigma_k\in [r]$ for $k=1,2$.
It is easy to see that
$L_{\sigma_1}\subset L_{\sigma_2}$ if
$\sigma_1\supset \sigma_2$.
Letting
%%()%%
\begin{equation*}
%%% 
Pr(\Sigma)=\{\sigma =\{i_1,\cdots ,i_s\} \subset [r]:
\{\textit{\textbf{n}}_{i_1},\cdots ,\textit{\textbf{n}}_{i_s}\}
\mbox{ is a primitive collection}\},
\end{equation*}
%%%
we see that
%%(6.2)%%
\begin{equation}
L(\Sigma)=\bigcup_{\sigma\in Pr(\Sigma)}L_{\sigma}
\end{equation}
%%%  
and by using (\ref{eq: rmin}) we obtain the equality 
%%%(6.3)%%
\begin{equation}\label{eq: dim rmin}
%%%
\dim L(\Sigma)=2(r-\rmin (\Sigma)).
\end{equation}
}
%%%
\end{rmk}
%%%(End of Remark 6.2)%%%%%

Our goal in this section is to construct, by means of the
{\it non-degenerate} simplicial resolution  of the discriminant, a spectral sequence converging to the homology of
$\Hol_D^*(S^2,\XS)$.

%%(Definition 6.3)%%%%
\begin{dfn}\label{non-degenerate simp.}
%%%
{\rm
Let 
$(\mathcal{X}^D,{\pi}^{\Delta}_D:\mathcal{X}^D\to\Sigma_D)$ 
%and
%$(\tilde{\SZ}(d),\ ^{\p}\tilde{\pi}_d^{\Delta}:\SZ(d)\to\Sigma_d^*)$ 
be the non-degenerate simplicial resolution associated to the surjective map
$\pi_D:Z_D\to \Sigma_D$ 
with the natural increasing filtration as in Definition \ref{def: def},
$$
\emptyset =
\SZ_0
\subset \SZ_1\subset 
\SZ_2\subset \cdots
\subset 
\SZ=\bigcup_{k= 0}^{\infty}\SZ_k.
$$
}
\end{dfn}
%%%(End of Definition 6.3)%%%%%%

%%
%%%%%%
%By Lemma \ref{lemma: simp} 
By Lemma \ref{lemma: simp},
the map
$\pi_D^{\Delta}:
\SZ\stackrel{\simeq}{\rightarrow}\Sigma_D$
is a homotopy equivalence which
extends to  a homotopy equivalence
%%%%%%%%%
$\pi_{D+}^{\Delta}:\SZ_+\stackrel{\simeq}{\rightarrow}{\Sigma_{D+}},$
%%%%%
where $X_+$ denotes the one-point compactification of a
locally compact space $X$.
%%
%\par
%%%
Since
${\mathcal{X}_k^{D}}_+/{\SZ_{k-1}}_+
\cong (\SZ_k\setminus \SZ_{k-1})_+$,
we have a spectral sequence 
%%%%%%%%%%%
%$$
$$
\big\{E_{t;D}^{k,s},
d_t:E_{t;D}^{k,s}\to E_{t;D}^{k+t,s+1-t}
\big\}
\Rightarrow
H^{k+s}_c(\Sigma_D,\Z),
$$
%$$
where
$E_{1;D}^{k,s}=\tilde{H}^{k+s}_c(\SZ_k\setminus\SZ_{k-1},\Z)$ and
$H_c^k(X,\Z)$ denotes the cohomology group with compact supports given by 
$
H_c^k(X,\Z)= H^k(X_+,\Z).
$
%%%%%%%%%%%%%
\par\vspace{2mm}\par
%%%%%
Let $N(D)$ and $N(\textit{\textbf{a}})$ denote the positive integers given by
%%(6.4)%%
\begin{equation}
%%%%%%%%
N(D)=\sum_{k=1}^r d_k,\quad
N(\textit{\textbf{a}})=\sum_{k=1}^ra_k.
\end{equation}
%%%%%%%%%
%%%
%%%
Since there is a homeomorphism
$\P^D\cong \C^{N(D)}$,
by Alexander duality  there is a natural
isomorphism

%%%(6.5)%%%
\begin{equation}\label{Al}
%%%%%%%%%
\tilde{H}_k(\Hol_d^*(S^2,\XS),\Z)\cong
\tilde{H}_c^{2N(D)-k-1}(\Sigma_D,\Z)
\quad
\mbox{for any }k.
\end{equation}
%%%
By
reindexing we obtain a
spectral sequence
%%

%%%(6.6)%%%
\begin{eqnarray}\label{SS}
%%%%%%%%%%%%%%%%%%%
&&\big\{E^{t;D}_{k,s}, \tilde{d}^{t}:E^{t;D}_{k,s}\to E^{t;D}_{k+t,s+t-1}
\big\}
\Rightarrow H_{s-k}(\Hol_d^*(S^2,\XS),\Z),
\end{eqnarray}
%%%%%%%
%if $s-k\leq 2nd-2$,
where
$E^{1;D}_{k,s}=
\tilde{H}^{2N(D)+k-s-1}_c(\SZ_k\setminus\SZ_{k-1},\Z).$
%and $\tilde{E}^t_{r,s}(d)=E_t^{r,N_d^*-1-s}(d).$
%%%%%%
\par\vspace{2mm}\par
%%%%
%Let $C_k(X)$ be the configuration space of unordered 
%$k$-distinct  points in $X$ given by the orbit space
%$C_k(X)=F(X,k)/S_k$.
%%%
Let
 $L_{k;\Sigma}\subset (\C\times L(\Sigma))^k$ denote the subspace
defined by
%%%%%()
%\begin{equation*}\label{l}
$$
L_{k;\Sigma}=\{((x_1,s_1),\cdots ,(x_k,s_k)): 
x_j\in \C,s_j\in L(\Sigma),
x_l\not= x_j\mbox{ if }l\not= j\}.
$$
%\end{equation*}
%%%%%%%%%
The symmetric group $S_k$ on $k$ letters  acts on $L_{k;\Sigma}$ by permuting coordinates. Let
$C_{k;\Sigma}$ denote the orbit space
%%%%%
%%%%(6.7)%%%%
\begin{equation}\label{Ck}
C_{k;\Sigma}=L_{k;\Sigma}/S_k.
\end{equation}
%%%%%%%%%%%
%%%%
Note that $C_{k;\Sigma}$ 
is a cell-complex of  dimension
$2(1+r-r_{\rm min}(\Sigma))k$
by (\ref{eq: dim rmin}).
%%

%%%%(Lemma 6.4)%%%%
\begin{lmm}\label{lemma: vector bundle*}
%%%%%%%%%%%%%%%%%%
%%%
If  
$1\leq k\leq d_{\rm min}=\min\{d_1,\cdots ,d_r\}$,
$\SZ_k\setminus\SZ_{k-1}$
is homeomorphic to the total space of a real affine
bundle $\xi_{D,k}$ over $C_{k;\Sigma}$ with rank 
$l_{D,k}=2N(D)-2rk+k-1$.
%%%%%%%%%%%%%%%%%%
\end{lmm}
%%%%%%%%(Proof of Lemma 6.4)%%%
\begin{proof}
%%%%%%%%%%%
The argument is exactly analogous to the one in the proof of  
\cite[Lemma 4.4]{AKY1}. 
%%%
%%%
Namely, an element of $\SZ_k\setminus\SZ_{k-1}$ is represented by 
$(F,u)=((f_1,\cdots ,f_{r}),u)$, where 
$F=(f_1,\cdots ,f_{r})$ is an 
$r$-tuple of monic polynomials in $\Sigma_D$ and $u$ is an element of the interior of
the span of the images of $k$ distinct points 
$\{x_1,\cdots, x_k\}\in C_k(\C)$ 
such that
%%%%%
$F(x_j)=(f_1(x_j),\cdots ,f_{r}(x_j))\in L(\Sigma)$ for each $1\leq j\leq k$, 
%%%%
under a suitable embedding.
%of $\C$ into Euclidean space.% satisfying the condition ({\ref{equ: filtration}}$)_k$.
%%%%
\ 
Since the $k$ distinct points $\{x_j\}_{j=1}^k$ 
are uniquely determined by $u$, by the definition of the non-degenerate simplicial resolution, %(cf.  ({\ref{equ: filtration}}$)_k$),
 there are projection maps
%%%%%%%%%%%% 
$\pi_{k,D} :\mathcal{X}^{D}_k\setminus
\mathcal{X}^{D}_{k-1}\to C_{k;\Sigma}$
%%%%%%%%%%%%
defined by
$((f_1,\cdots ,f_{r}),u) \mapsto 
\{(x_1,F(x_1)),\dots, (x_k,F(x_k))\}$. 

\par
%%%
%%%%%(Fiber of pi_k)%%%%%%
Now suppose that $1\leq k\leq d_{\rm min}$.
Let $c=\{(x_j,s_j)\}_{j=1}^k\in C_{k;\Sigma}$
$(x_j\in \C$, $s_j\in L(\Sigma))$ be any fixed element and consider the fibre  $\pi_{k,D}^{-1}(c)$.
%%%
For each $1\leq j\leq k$,
we set $s_j=(s_{1,j},\cdots ,s_{r,j})$ and
consider the condition  
%%%(6.8)%%%
\begin{equation}\label{equ: pik}
%%%%%%%%%%%
F(x_j)=(f_1(x_j),\cdots ,f_{r}(x_j))=s_j
\quad
\Leftrightarrow
\quad
f_t(x_j)=s_{t,j}
\quad
\mbox{for }1\leq t\leq r.
\end{equation}
%%%%%%%%%%%
%%
In general, %for each $1\leq t\leq r$,
the condition $f_t(x_j)=s_{t,j}$ gives
one  linear condition on the coefficients of $f_t$,
and determines an affine hyperplane in $\P^{d_t}(\C)$. 
%%%
For example, if we set $f_t(z)=z^{d_t}+\sum_{i=0}^{d_t-1}a_{i,t}z^{i}$,
then
$f_t(x_j)=s_{t,j}$ for any $1\leq j\leq k$
if and only if
%%%%%%
%%(6.9)(matrix equation)%%
\begin{equation}\label{equ: matrix equation}
%%%%%%%%%
\begin{bmatrix}
1 & x_1 & x_1^2 & \cdots & x_1^{d_t-1}
\\
1 & x_2 & x_2^2 & \cdots & x_2^{d_t-1}
\\
\vdots & \ddots & \ddots & \ddots & \vdots
%\\
%1 & x_{k-1} & x_{k-1}^2 & \cdots & x_{k-1}^{d-1}
\\
1 & x_k & x_k^2 & \cdots & x_k^{d_t-1}
\end{bmatrix}
%%%%
\cdot
\begin{bmatrix}
a_{0,t}\\ a_{1,t} \\ \vdots %\\ a_{2,t} 
\\ a_{d_t-1,t}
\end{bmatrix}
=
\begin{bmatrix}
s_{t,1}-x_1^{d_t}\\ s_{t,2}-x_2^{d_t} \\ \vdots %\\ s_{t,k-1}-x_{k-1}^d 
\\ s_{t,k}-x_k^{d_t}
\end{bmatrix}
\end{equation}
%%%%
%%%%
Since $1\leq k\leq d_{\rm min}$ and
 $\{x_j\}_{j=1}^k\in C_k(\C)$,  
it follows from the properties of Vandermonde matrices that the condition (\ref{equ: matrix equation}) 
gives exactly $k$ independent conditions on the coefficients of $f_t(z)$.
Thus  the space of polynomials $f_t(z)$ in $\P^{d_t}(\C)$ which satisfies
(\ref{equ: matrix equation})
is the intersection of $k$ affine hyperplanes in general position
and has codimension $k$ in $\P^{d_t}(\C)$.
%%%%%
Hence,
the fibre $\pi_{k,D}^{-1}(c)$ is homeomorphic  to the product of an open $(k-1)$-simplex
 with the real affine space of dimension
 $2\sum_{i=1}^r(d_i-k)=2N(D)-2rk$.
It is now easy to show that  $\pi_{k,D}$ is a (locally trivial) real affine bundle over $C_{k;\Sigma}$ of rank $l_{D,k}
=2N(D)-2rk+k-1$.
\end{proof}
%%(End of proof of Lemma 6.4)%%%

%%%%%%(Lemma 6.5)%%
\begin{lmm}\label{lemma: E11}
%%%%%%
If $1\leq k\leq  d_{\rm min}$, there is a natural isomorphism
$$
E^{1;D}_{k,s}\cong
\tilde{H}^{2rk-s}_c(C_{k;\Sigma},\pm \Z),
$$
where 
the twisted coefficients system $\pm \Z$  comes from
the Thom isomorphism.
\end{lmm}
%%%%
\begin{proof}
%%%(Proof of Lemma 6.5)%%
Suppose that $1\leq k\leq d_{\rm min}$.
%%%
By Lemma \ref{lemma: vector bundle*}, there is a
homeomorphism
$
(\SZ_k\setminus\SZ_{k-1})_+\cong T(\xi_{D,k}),
$
where $T(\xi_{D,k})$ denotes the Thom space of
%one-point compactification of
$\xi_{D,k}$.
%%%%%%%
Since $(2N(D)+k-s-1)-l_{D,k}
=
2rk-s,$
%$$
%%%%%
by using the Thom isomorphism 
there is a natural isomorphism 
%%%
$
E^{1;d}_{k,s}
\cong 
\tilde{H}^{2nd+k-s-1}(T(\xi_{d,k}),\Z)
\cong
\tilde{H}^{2rk-s}_c(C_{k;\Sigma},\pm \Z).
$
\end{proof}
%%%(End of proof of Lemma 6.5)%%%%%%
%%

%%%(Definition 6.6)%%%
\begin{dfn}
%%%%%%%%%%%%%%%%%%%%%%
{\rm
For  an $r$-tuple 
$E=(e_1,\cdots ,e_r)\in (\Z_{\geq 1})^r$
of positive integers,
let $N(E)$ denote the positive integer
$N(E)=\sum_{k=1}^re_k$ and let
$U_E=\{w\in \C:\mbox{Re}(w)<N(E)\}$
as in Definition \ref{dfn: stabilization etc}.
\qed
%%%
}
%%%%%%
\end{dfn}
%%%(End of Definition 6.6)%%

%%
Consider
the stabilization map
$s_D:\Hol_d^*(S^2,\XS)\to \Hol^*_{D+\textit{\textbf{a}} }(S^2,\XS)$
of (\ref{eq: sD}).
%%%
%%
It is easy to see that it extends to an open embedding
%%
%%%(6.10)%%
\begin{equation}\label{equ: sssd}
%%%%%%%%%
\overline{s}_D:\C^{N(\textit{\textbf{a}})}\times 
\Hol_D^*(S^2,\XS)\to
\Hol_{D+\textit{\textbf{a}}}^*(S^2,\XS).
\end{equation}
%%%%%%%% 
%%%
Moreover, it also naturally extends to an open embedding
$\tilde{s}_D:\P^D\to \P^{D+\textit{\textbf{a}}}$ and  by  restriction  we obtain an open embedding
%%%%
%%%()%%
%\begin{equation}\label{equ: open embedding}
%%%%%%%
$\tilde{s}_D:\C^{N(\textit{\textbf{a}})}\times \Sigma_D\to 
\Sigma_{D+\textit{\textbf{a}}}.$
%\end{equation}
%%%
Since one-point compactification is contravariant for open embeddings,
this map induces a map
$\tilde{s}_{D+}:(\Sigma_{D+\textit{\textbf{a}}})_+
\to
(\C^{N(\textit{\textbf{a}})}\times \Sigma_D)_+=S^{2N(\textit{\textbf{a}})}\wedge \Sigma_{D+}.$
\par
Note that there is a commutative diagram
%%%
%%%%%(6.11)%%
\begin{equation}\label{eq: diagram}
%%%%%%%%%%%%%
\begin{CD}
\tilde{H}_k(\Hol_D^*(S^2,\XS),\Z) @>{s_D}_*>>
\tilde{H}_k(\Hol_{D+\textit{\textbf{a}}}^*(S^2,\XS),\Z)
\\
@V{Al}V{\cong}V @V{Al}V{\cong}V
\\
\tilde{H}^{2N(D)-k-1}_c(\Sigma_D,\Z)
@>{\tilde{s}_{D+}}^{\ *}>>
\tilde{H}^{2(N(D)+N(\textit{\textbf{a}}))-k-1}_c(\Sigma_{D+\textit{\textbf{a}}},\Z)
\end{CD}
\end{equation}
%%%% 
where $Al$ is the Alexander duality isomorphism and
 ${\tilde{s}_{D+}}^{\ *}$ denotes the composite of 
%homomorphisms 
the suspension isomorphism with the homomorphism
${(\tilde{s}_{D+})^*}$,
$$
%\begin{CD}
\tilde{H}^{M}_c(\Sigma_D,\Z)
\stackrel{\cong}{\rightarrow}
\tilde{H}^{M+2N(\textit{\textbf{a}})}_c
(\C^{N(\textit{\textbf{a}})}\times \Sigma_D,\Z)
%\stackrel{(\tilde{s}_{d+})^*}{\longrightarrow}
\stackrel{(\tilde{s}_{D+})^*}{\longrightarrow}
\tilde{H}^{M+2N(\textit{\textbf{a}})}_c(\Sigma_{D+\textit{\textbf{a}}},\Z),
%\end{CD}
$$
where $M=2N(D)-k-1$.
%%%%%
By the universality of the non-degenerate simplicial resolution
(\cite[pages 286-287]{Mo2}), 
the map $\tilde{s}_D$ also naturally extends to a filtration preserving open embedding
%%%()%%%%%
%\begin{equation}\label{equ: flitr-preserve map}
$\tilde{s}_D:\C^{N(\textit{\textbf{a}})} \times \SZ \to \SZd$
%\end{equation}
%%%%%%%%%%%%%%
between non-degenerate simplicial resolutions.
This  induces a filtration preserving map
$(\tilde{s}_D)_+:\SZd_+\to (\C^{N(\textit{\textbf{a}})} \times \SZ)_+
=S^{2N(\textit{\textbf{a}})}\wedge \SZ_+$,
and thus a homomorphism of spectral sequences
%%%
%%%(6.12)%%%%
\begin{equation}\label{equ: theta1}
%%%%%%%%%%%%
\{ \tilde{\theta}_{k,s}^t:E^{t;D}_{k,s}\to 
E^{t;D+\textit{\textbf{a}}}_{k,s}\},
\quad \mbox{where}
%%%%%
\end{equation}
%%%
%%%%%%%%
\begin{eqnarray*}
%%%%%%%
\big\{E^{t;D}_{k,s}, \tilde{d}^{t}:
E^{t;D}_{k,s}\to E^{t;D}_{k+t,s+t-1}
\big\}
\quad &\Rightarrow &
 H_{s-k}(\Hol_{D}^*(S^2,\XS),\Z),
\\
\big\{E^{t;D+\textit{\textbf{a}}}_{k,s}, \tilde{d}^{t}:
E^{t;D+\textit{\textbf{a}}}_{k,s}\to 
E^{t;D+\textit{\textbf{a}}}_{k+t,s+t-1}
\big\}
\quad&\Rightarrow& 
 H_{s-k}(\Hol_{D+\textit{\textbf{a}}}^*(S^2,\XS),\Z),
%%%
\end{eqnarray*}
%%%
$$
E^{1;D}_{k,s}=
\tilde{H}_c^{2N(D)+k-1-s}(\SZ_k\setminus \SZ_{k-1}),
E^{1;D+\textit{\textbf{a}}}_{k,s}
=
\tilde{H}_c^{2N(D+\textit{\textbf{a}})+k-1-s}
(\mathcal{X}_k^{D+\textit{\textbf{a}}}\setminus 
\mathcal{X}_{k-1}^{D+\textit{\textbf{a}}}).
$$

%%%
%%%
%%%(Lemma 6.7)%%%
\begin{lmm}\label{lmm: E1}
%%%%%%%%%%%%%%%%%
If $0\leq k\leq d_{\rm min}$, 
$\tilde{\theta}^1_{k,s}:E^{1;D}_{k,s}\to 
E^{1;D+\textit{\textbf{a}}}_{k,s}$ is
an isomorphism for any $s$.
\end{lmm}
%%%%%%%%%%%%%%%%
\begin{proof}
%%%%%%%%%%%%%%%%
Since the case $k=0$ is clear,
suppose that $1\leq k\leq d_{\rm min}$.
It follows from the proof of Lemma \ref{lemma: vector bundle*}
that there is a homotopy commutative diagram of affine vector bundles
$$
\begin{CD}
%\C^n\times 
\C^n\times (\SZ_k\setminus\SZ_{k-1}) @>>> C_{k;\Sigma}
\\
@VVV \Vert @.
\\
\SZd_k\setminus \SZd_{k-1} @>>> C_{k;\Sigma}
\end{CD}
$$
Hence, %by the naturality of Thom isomorphisms 
we have 
a commutative diagram
$$
\begin{CD}
E^{1,D}_{k,s} @>>\cong> \tilde{H}^{2rk-s}_c(C_{k;\Sigma},\pm \Z)
\\
@V{\tilde{\theta}_{k,s}^1}VV \Vert @.
\\
E^{1,D+\textit{\textbf{a}}}_{k,s} @>>\cong> \tilde{H}^{2rk-s}_c(C_{k;\Sigma},\pm \Z)
\end{CD}
$$
%where %$r_{\rm min}=r_{\rm min}(I)$ and 
%$T$ denotes the Thom isomorphism.
and the assertion follows.
%%%%%%%%
\end{proof}
%%(End of proof of Lemma 6.7)%%%%
%%%%%%%%%%%%%%%
%%

Now we consider the spectral sequences induced by the 
truncated simplicial resolutions.

%%%(Definition 6.8)%%%
\begin{dfn}
%%%%%%%%%%%%%%%%%%%%%%
{\rm
Let $X^{\Delta}$ denote the truncated 
(after the $d_{\rm min}$-th term) simplicial resolution of $\Sigma_D$
with the natural filtration
$
\emptyset =X^{\Delta}_0\subset
X^{\Delta}_1\subset \cdots\subset
X^{\Delta}_{d_{\rm min}}\subset X^{\Delta}_{d_{\rm min}+1}=
X^{\Delta}_{d_{\rm min}+2}=\cdots =X^{\Delta},
$
where $X^{\Delta}_k=\SZ_k$ if $k\leq d_{\rm min}$ and 
$X^{\Delta}_k=X^{\Delta}(d_{\rm min})$ if $k\geq d_{\rm min}+1$.
%%%%%%
\par\vspace{1mm}\par
%%%%
Similarly,
let $Y^{\Delta}$ denote  truncated (after the $d_{\rm min}$-th term) simplicial resolution of 
$\Sigma_{D+\textit{\textbf{a}}}$
with the natural filtration
$
\emptyset =Y^{\Delta}_0\subset
Y^{\Delta}_1\subset \cdots\subset
Y^{\Delta}_{d_{\rm min}}\subset Y^{\Delta}_{d_{\rm min}+1}
=Y^{\Delta}_{d_{\rm min}+2}=\cdots =Y^{\Delta},
$
where $Y^{\Delta}_k=\SZd_k$ if $k\leq d_{\rm min}$ and 
$Y^{\Delta}_k=Y^{\Delta}$ if $k\geq d_{\rm min}+1$.
\qed
%%%%%%%
}
%%%%%%%
\end{dfn}
%%%%%%(End of Definition 6.8)%%%%

%%%%
By using Lemma \ref{Lemma: truncated} and the same method
as in \cite[\S 2 and \S 3]{Mo3} (cf. \cite[Lemma 2.2]{KY4}), 
we obtain the following {\it  truncated spectral sequences}
%%%%()%%%%%%%%
\begin{eqnarray*}\label{equ: spectral sequ2}
%%%%%%%%%%%%%%%%
\big\{E^{t}_{k,s}, d^{t}:E^{t}_{k,s}\to 
E^{t}_{k+t,s+t-1}
\big\}
&\Rightarrow& H_{s-k}(\Hol_{D}^*(S^2,\XS),\Z),
%%%
\\
%%%
\big\{\ 
^{\p}E^{t}_{k,s},\  d^{t}:\ ^{\p}E^{t}_{k,s}\to 
\  ^{\p}E^{t}_{k+t,s+t-1}
\big\}
&\Rightarrow& H_{s-k}(\Hol_{D+\textit{\textbf{a}}}^*(S^2,\XS),\Z),
%%%%
\end{eqnarray*}
%%%
$$
E^{1}_{k,s}=\ \tilde{H}_c^{2N(D)+k-1-s}(X^{\Delta}_k\setminus X^{\Delta}_{k-1}),
\ 
^{\p}E^{1}_{k,s}=\ \tilde{H}_c^{2(N(D)+N(\textit{\textbf{a}}))+k-1-s}(Y^{\Delta}_k\setminus Y^{\Delta}_{k-1}).
$$
%%%
By the naturality of truncated simplicial resolutions,
the filtration preserving map
$\tilde{s}_D:\C^{N(\textit{\textbf{a}})}\times \SZ \to \SZd$
  gives rise to a natural filtration preserving map
%%()%%
%\begin{equation}
%%%%%%%%
$\tilde{s}_D^{\p}:\C^{N(\textit{\textbf{a}})}\times X^{\Delta} \to Y^{\Delta}$
%\end{equation}
%%%%%%%% 
which, in a way analogous to  (\ref{equ: theta1}), induces
a homomorphism of spectral sequences 
%%%%%%%%%%%
%%%(6.13)%%%%
\begin{equation}\label{equ: theta2}
%%%%%%%%%%%%
\{ \theta_{k,s}^t:E^{t}_{k,s}\to \ ^{\p}E^{t}_{k,s}\}.
\end{equation}
%%%%%%%%%%

%%%%
%%%%(Lemma 6.9)%%
\begin{lmm}\label{lmm: Ed}
%%%%%%%%%%%%%%%%%
%Let $r_{\rm min}=r_{\rm min}(I)$, and let $\epsilon \in\{0,1\}$.
%%%%%%%%
\begin{enumerate}
%%(i)%%%
\item[$\I$]
If $k<0$ or $k\geq \dmin+2$,
$E^1_{k,s}=\ ^{\p}E^1_{k,s}=0$ for any $s$.
%%(ii)%%
\item[$\II$]
$E^1_{0,0}=\ ^{\p}E^1_{0,0}=\Z$ and $E^1_{0,s}=\ ^{\p}E^1_{0,s}=0$ if $s\not= 0$.
%%(iii)%%%
\item[$\III$]
If $1\leq k\leq d_{\rm min}$, there are  
isomorphisms
$
E^1_{k,s}\cong \ ^{\p}E^1_{k,s}\cong \tilde{H}^{2rk-s}_c(C_{k;\Sigma},\pm \Z).
$
%%(iv)%%%
\item[$\IV$]
$E^1_{d_{\rm min}+1,s}=\ ^{\p}E^1_{d_{\rm min}+1,s}=0$ for any 
$s\leq (2\rmin (\Sigma)-2)d_{\rm min}-1$.
%%%%%%%%%%
\end{enumerate}
%%%%%%
%%%%%
\end{lmm}
%%%%%%
\begin{proof}
%%%%%
Let us write $\rmin =\rmin (\Sigma)$.
Since the proofs of both cases are identical,  it suffices to prove the assertions for the case
$E^1_{k,s}$.
Since $X^{\Delta}_k=\SZ_k$ for any $k\geq d_{\rm min}+2$,
the assertions (i) and (ii) are clearly true.
Since $X^{\Delta}_k=\SZ_k$ for any $k\leq d_{\rm min}$,
(iii) easily follows from Lemma \ref{lemma: E11}.
Thus it remains to prove (iv).
%%%%
By Lemma \cite[Lemma 2.1]{Mo3},
%%%%%%%
\begin{align*}
%%%%%%
\dim (X^{\Delta}_{d_{\rm min}+1}\setminus 
X^{\Delta}_{d_{\rm min}})
&=
\dim (\SZ_{d_{\rm min}}\setminus \SZ_{d_{\rm min}-1})+1
=(l_{D,d_{\rm min}}+\dim C_{d_{\rm min};\Sigma})+1
\\
&=2N(D)+3d_{\rm min}-2\rmin d_{\rm min}.
\end{align*}
%%%
Since 
$E^1_{d_{\rm min}+1,s}=\tilde{H}_c^{2N(D)+d_{\rm min}-s}
(X^{\Delta}_{d_{\rm min}+1}\setminus X^{\Delta}_{d_{\rm min}},\Z)$
and
$2N(D)+d_{\rm min}-s>\dim (X^{\Delta}_{d+1}\setminus X^{\Delta}_d)$
$\Leftrightarrow$
$s\leq (2\rmin -2)d_{\rm min}-1$,
we see that
$E^1_{d_{\rm min}+1,s}=0$ for any $s\leq (2\rmin -2)d_{\rm min}-1$.
%%%%
\end{proof}
%%%(End of proof of Lemma 5.9)%%%%%%%
%A similar computation also proves the following.
%%%%

%%%(Lemma 6.10)%%
\begin{lmm}\label{lmm: E2}
%%%%%%%%%%%%%%%%%
If $0\leq k\leq d_{\rm min}$, $\theta^1_{k,s}:E^{1}_{k,s}\to \ ^{\p}E^{1}_{k,s}$ is
an isomorphism for any $s$.
\end{lmm}
%%%%%%%%%%%%%%%%
\begin{proof}
Since $(X^{\Delta}_k,Y^{\Delta}_k)=(\SZ_k,\SZd_k)$ for $k\leq \dmin$,
the assertion follows from Lemma \ref{lmm: E1}.
\end{proof}
%%%%%%%%%%%%%%
%%

%%%(SECTION 7)%%%
\section{The proofs of the main results}\label{section: proofs}
%%%%%%%%%%%%%

In this section we prove Theorem \ref{thm: I} and
Corollary \ref{crl: I}.
For this purpose,  from now on we always assume that
 $\XS$ is a smooth toric variety such that 
the conditions $($\ref{equ: homogenous}.1$)$ and 
$($\ref{equ: homogenous}.2$)$ are satisfied.
First we prove the following key unstable result.

%%%%%(Theorem 7.1)%%
\begin{thm}\label{thm: III}
%%%%%%%%%%%%%%%%%%%
The stabilization map
$s_D:\Hol_D^*(S^2,\XS)\to \Hol_{D+\textit{\textbf{a}}}^*(S^2,\XS)$
is a homology equivalence through dimension
$d(D,\Sigma)=(2\rmin (\Sigma )-3)\dmin -2$.
%%%%%%%%%%%%%%%%%
\end{thm}
%%%%%%%%%%%%%%%%%%
\begin{proof}
%%%%%%%%%%%%%%%%%
We write $\rmin =\rmin (\Sigma)$, and
consider the homomorphism
$\theta_{k,s}^t:E^{t}_{k,s}\to \ ^{\p}E^{t}_{k,s}$
of truncated spectral sequences given in (\ref{equ: theta2}).
%%%%
By using the commutative diagram (\ref{eq: diagram}) and the comparison theorem for spectral sequences, 
it suffices to prove that the positive integer $d(D,\Sigma)$ 
has the 
following property:
%%%
\begin{enumerate}
\item[$(*)$]
$\theta^{\infty}_{k,s}$
is  an isomorphism for all $(k,s)$ such that $s-k\leq d(D,\Sigma)$.
\end{enumerate}
%%
%%
%%(The case r<0 or r \leq d+2)%%%%%
By Lemma \ref{lmm: Ed}, 
$E^1_{k,s}=\ ^{\p}E^1_{k,s}=0$ if
$k<0$, or if $k\geq \dmin +2$, or if $k=\dmin +1$ with 
$s\leq (2\rmin -2)\dmin -1$.
Since $(2\rmin -2)\dmin -1-(\dmin+1)=(2\rmin -3)\dmin -2
=d(D,\Sigma)$,
we  see that:
%()%%
%%%%%%%
%%%%%%%
\begin{enumerate}
%%%%%%%
\item[$(*)_1$]
if $k< 0$ or $k\geq \dmin +1$,
$\theta^{\infty}_{k,s}$ is an isomorphism for all $(k,s)$ such that
$s-k\leq d(D,\Sigma)$.
\end{enumerate}
\par
Next, we assume that $0\leq k\leq \dmin$, and investigate the condition that
$\theta^{\infty}_{k,s}$  is an isomorphism.
%%%
%Then
Note that the groups $E^1_{k_1,s_1}$ and $^{\p}E^1_{k_1,s_1}$ are not known for
%%(S_1)%%%
$(u,v)\in\mathcal{S}_1=
\{(\dmin+1,s)\in\Z^2:s\geq (2\rmin -2)\dmin \}$.
%%%%
%%%%
By considering the differentials $d^1$'s of
$E^1_{k,s}$ and $^{\p}E^1_{k,s}$,
%$d^1:\E^1_{k,s}\to \E^{1}_{k+1,s}$
%and
%$d^1:\Ed^1_{k,s}\to \Ed^{1}_{k+1,s}$
and applying Lemma \ref{lmm: E2}, we see that
$\theta^2_{k,s}$ is an isomorphism if
$(k,s)\notin \mathcal{S}_1 \cup \mathcal{S}_2$, where
%%%
%%(S_2)%%%
$$
\mathcal{S}_2=
\{(u,v)\in\Z^2:(u+1,v)\in \mathcal{S}_1\}
=\{(\dmin ,v)\in \Z^2:v\geq (2\rmin -2)\dmin\}.
$$
%%%%
%%%
%%%%
A similar argument  shows that
$\theta^3_{k,s}$ is an isomorphism if
%%(S_3)%%
$(k,s)\notin \bigcup_{t=1}^3\mathcal{S}_t$, where
$\mathcal{S}_3=\{(u,v)\in\Z^2:(u+2,v+1)\in \mathcal{S}_1\cup
\mathcal{S}_2\}.$
%%%%%
%\par
%%%%
Continuing in the same fashion,
considering the differentials
$d^t$'s on $E^t_{k,s}$ and $^{\p}E^t_{k,s}$
and applying the inductive hypothesis,
%Lemma \ref{lmm: E2}, 
we  see that $\theta^{\infty}_{k,s}$ is an isomorphism
if $\dis (k,s)\notin \mathcal{S}:=\bigcup_{t\geq 1}\mathcal{S}_t
=\bigcup_{t\geq 1}A_t$,
where  $A_t$ denotes the set
%%%%(Def. of A_t)%%%%%%%
$$
A_t=
\left\{
\begin{array}{c|l}
 &\mbox{ There are positive integers }l_1,\cdots ,l_t
\mbox{ such that},
\\
(u,v)\in \Z^2 &\  1\leq l_1<l_2<\cdots <l_t,\ 
u+\sum_{j=1}^tl_j=\dmin +1,
\\
& \ v+\sum_{j=1}^t(l_j-1)\geq (2\rmin -2)\dmin
\end{array}
\right\}.
$$
%%%%%%%% 
Note that 
if this set was empty for every $t$, then, of course, the conclusion of 
Theorem \ref{thm: III} would hold in all dimensions (this is known to be false in general). 
%\par
If $\dis A_t\not= \emptyset$, it is easy to see that
%%%
$$
a(t)=\min \{s-k:(k,s)\in A_t\}=
(2\rmin -2)\dmin -(\dmin +1)+t
=d(D,\Sigma)+t+1.
$$
%%%
Hence, we obtain that
$\min \{a(t):t\geq 1,A_t\not=\emptyset\}=d(D,\Sigma)+2.$
%%%
 Since $\theta^{\infty}_{k,s}$ is an isomorphism
for any $(k,s)\notin \bigcup_{t\geq 1}A_t$ for each $0\leq k\leq \dmin$,
we have the following:
%%(*2)%%%
\begin{enumerate}
%%%%%%%%
\item[$(*)_2$]
If $0\leq k\leq \dmin$,
$\theta^{\infty}_{k,s}$ is  an isomorphism for any $(k,s)$ such that
$s-k\leq  d(D,\Sigma)+1.$
\end{enumerate}
%%%%%%
Then, by $(*)_1$ and $(*)_2$, we know that
$\theta^{\infty}_{k,s}:E^{\infty}_{k,s}\stackrel{\cong}{\rightarrow}
\ ^{\p}E^{\infty}_{k,s}$ is an isomorphism for any $(k,s)$
if $s-k\leq d(D,\Sigma)$. This completes the proof
of Theorem \ref{thm: III}.
%%%%
%%%%
\end{proof}
%%(End of proof of Theorem 7.1)%%
%
%%(Corollary 7.2)%%
\begin{crl}\label{crl: homology iso}
%%%%%%%%
The inclusion map
$i_D:\Hol_D^*(S^2,\XS)\to \Omega^2_D\XS$
is a homology equivalence through dimension
$d(D,\Sigma)$.
%%%%%%
\end{crl}
%%%%%%
\begin{proof}
%%%
The assertion easily follows from Theorem \ref{thm: II}
and  Theorem \ref{thm: III}.
\end{proof}
%%%%%
%%
%%
%%
%%%(Lemma 7.3)%%
\begin{lmm}\label{lmm: connectedness}
%%%%%%%%%%%%%%%
The space $\Omega^2_D\XS$ is
$2(\rmin (\Sigma)-2)$-connected.
%%%%%%%%%%%%%%
\end{lmm}
%%%%%%%%%%%%%%
\begin{proof}
%%%%%%%%%%%%%
Since $\XS$ is a smooth toric variety, it follows from
\cite[Prop. 6.7]{Pa1}
that the group $\GS$ acts on $U(\mathcal{K}_{\Sigma})$ freely and
there is a fibration sequence
$\T^r_{\C}\to U(\mathcal{K}_{\Sigma})\to \XS$.
Hence, there is a homotopy equivalence
$\Omega^2_D\XS\simeq \Omega^2U(\mathcal{K}_{\Sigma})$.
On the other hand, it follows from \cite[Lemma 5.4 and Remark 5.5]{KOY1}
that $U(\mathcal{K}_{\Sigma})$ is
$2(\rmin (\Sigma)-1)$-connected.
Thus, the space $\Omega^2_D\XS$ is
$2(\rmin (\Sigma)-2)$-connected.
%%%%%
\end{proof}
%%%(End of proof of Lemma 7.3)%%%
%%
%%
%%
%%
%%(Lemma 7.4)%%
\begin{lmm}\label{lmm: 1-connected}
%%%%%%%%%%%%%%%%%%%%%%%%%%%%%%%%%%%%
The space
$\Hol_D^*(S^2,\XS)$ is simply connected
if $\rmin (\Sigma)\geq 3$
and  the group
$\pi_1(\Hol_D^*(S^2,\XS))$ is an abelian group
if $\rmin (\Sigma)=2$.
%%%%%%%%%%%%%%%%%%%%%%%%%%%%%%%%%%%
\end{lmm}
%%%%%%%%%%%%%%%%%%%%%%%%%%%%%%%%%%%
%%(Proof of Lemma 7.4)%%
\begin{proof}
%%%%%%
Suppose that $\rmin (\Sigma)\geq 2$.
By using the string representation of elements of
$\pi_1(\Hol_D^*(S^2,\XS))$ as in \cite[\S Appendix]{GKY1},
one can show that  $\pi_1(\Hol_D^*(S^2,\XS))$ is an abelian group
and the case $\rmin (\Sigma)=2$ was obtained.
Now assume that $\rmin (\Sigma)\geq 3$.
By Hurewicz Theorem, the Hurewicz homomorphism
$h:\pi_1(\Hol_D^*(S^2,\XS))\stackrel{\cong}{\rightarrow} H_1(\Hol_D^*(S^2,\XS),\Z)$ is an isomorphism.
Note that $\Omega^2_D\XS$ is at least $2$-connected
(by Lemma \ref{lmm: connectedness}).
Since
$d(D,\Sigma)=(2\rmin (\Sigma)-3)\dmin -2\geq 3\dmin-2\geq 1$, 
by  Corollary \ref{crl: homology iso}
the map $i_D$ induces the isomorphism
$
i_{D*}:
H_1(\Hol_D^*(S^2,\XS),\Z)
\stackrel{\cong}{\rightarrow}
H_1(\Omega^2_D\XS,\Z)=0$.
Hence, $\pi_1(\Hol_D^*(S^2,\XS))=0$
and the case $\rmin (\Sigma)\geq 3$ was also obtained.
\end{proof}
%%(End of proof of Lemma 7.4)%%
%
%
%
\par\vspace{1mm}\par

Now we can prove the main results 
(Theorem \ref{thm: I} and Corollary \ref{crl: I}).
%%
%%
%%
%%%%%%%%%%%%%%%%%%%%%%%%%%%%%%%%
%%%(Proof of Theorem 1.9)%%
\begin{proof}[Proof of Theorem \ref{thm: I}]
%%%%%%%%%%%%%%%%%%%%%%%%%%%%%%%%
By Corollary \ref{crl: homology iso},
it remains to prove that $i_D$ is a homotopy equivalence through dimension
$d(D,\Sigma)$ if $\rmin (\Sigma)\geq 3$.
Assume that $\rmin (\Sigma)\geq 3$.
Note that two spaces $\Hol_D^*(S^2,\XS)$ and $\Omega^2_D\XS$ are simply connected
by Lemma \ref{lmm: connectedness} and Lemma \ref{lmm: 1-connected}.
Hence, the map $i_D$ is a homotopy equivalence through dimension $d(D,\Sigma)$.
%%%%%%%%%%%%%%%%%%%%%%%%%%%%%%%
\end{proof}
%%%%%%%%(End of Theorem 1.9)%%%%%%
%
%
%
%
%
%%%
%%(Proof of Corollary 1.11)%%%%%%%%
\begin{proof}[Proof of Corollary \ref{crl: I}.]
%%%%%%%%%%%%%%%%%%%%%%%%%%%%%%%%%%%
Let $\XS$ be a compact smooth toric variety such that
$\Sigma (1)=\{\mbox{Cone}(\textit{\textbf{n}}_k):1\leq k\leq r\}$, 
where $\{\textit{\textbf{n}}_k\}_{k=1}^r$ are  primitive generators
as in \S 1.
Since $\XS$ is a compact, by (ii) of Lemma \ref{lmm: toric} we easily see that
the condition (\ref{equ: homogenous}.1) 
is satisfied for $\XS$.
Since $\Sigma_1\subsetneqq \Sigma$, by using
(i) of Lemma \ref{lmm: toric}
we  see that $X_{\Sigma_1}$ is a non-compact smooth toric subvariety of $\XS$.
Moreover, since
 $\Sigma (1)\subset \Sigma_1\subsetneqq \Sigma$, we see that
$\Sigma_1(1)=\Sigma (1)$.
Hence,  
the condition (\ref{equ: homogenous}.1) holds for $X_{\Sigma_1}$, too. 
%Note that the condition  (\ref{equ: homogenous}.2) also holds for ${X}_{\Sigma_1}$
%by the assumption.
Thus,
the assertion follows from Theorem \ref{thm: I}.
%%%
\end{proof}
%%(End of proof of Corollary 1.11)%%%%
%%
%%
%%
%%(Corollary 7.5)%%
\begin{crl}\label{crl: pi1}
%%%%%%%%%%%%%%%%%%
%%(i)%%
$\I$
%%%%%%%%
If $\rmin (\Sigma)\geq 3$,
the space $\Hol_D^*(S^2,\XS)$ is
$2(\rmin (\Sigma)-2)$-connected.
\par
%%(ii)%%
$\II$ 
%%%%%%%
If $\rmin(\Sigma)=2$ and 
$\dmin \geq 3$,
%the induced homomorphism
$i_{D*}:\pi_1(\Hol_D^*(S^2,\XS))
\stackrel{\cong}{\rightarrow}\pi_1(\Omega^2_D\XS)$ is an
isomorphism.
%%%
\end{crl}
%%%
\begin{proof}
%%%%
Since the first assertion easily follows from Lemma \ref{lmm: connectedness}
and Theorem \ref{thm: I}, it suffices to prove the assertion (ii).
Now assume that $\rmin (\Sigma)=2$.
Then by Lemma \ref{lmm: 1-connected}, the Hurewicz homomorphism
$h:\pi_1(\Hol_D^*(S^2,\XS))
\stackrel{\cong}{\rightarrow} 
H_1(\Hol_D^*(S^2,\XS),\Z)$ is an isomorphism.
Consider the commutative diagram
$$
\begin{CD}
\pi_1(\Hol_D^*(S^2,\XS)) @>{i_{D*}}>> \pi_1(\Omega^2_D\XS)
\\
@V{h}V{\cong}V @V{h}V{\cong}V
\\
H_1(\Hol_D^*(S^2,\XS),\Z) @>{i_{D*}}>{\cong}> H_1(\Omega^2_D\XS,\Z)
\end{CD}
$$
If $\dmin \geq 3$, $d(D,\Sigma)=\dmin -2\geq 1$.
Hence, by Theorem \ref{thm: I}, the induced homomorphism $i_{D*}$
on the homology is an isomorphism and the assertion (ii) follows
from the above commutative diagram.
%%%
\end{proof}
%%(End of proof of Corollary 7.5)%%

%%%%%%%%%%%%%%%%%%%%%%%%%%%%%%%%%%%%%%%%%%%%%%%%%%

%%()%%%%%%
\par\vspace{2mm}\par
\noindent{\bf Acknowledgements. }
%%%%%%%%
The authors should like to take this opportunity to thank  
Professor Masahiro Ohno 
for his many valuable  insights and suggestions concerning toric varieties.
%%%
%\par
The second author was supported by 
JSPS KAKENHI Grant Number 26400083.

%% The Appendices part is started with the command \appendix;
%% appendix sections are then done as normal sections
%% \appendix

%% \section{}
%% \label{}

%% References
%%
%% Following citation commands can be used in the body text:
%% Usage of \cite is as follows:
%%   \cite{key}          ==>>  [#]
%%   \cite[chap. 2]{key} ==>>  [#, chap. 2]
%%   \citet{key}         ==>>  Author [#]

%% References with bibTeX database:

%\bibliographystyle{model1-num-names}
%\bibliography{<your-bib-database>}

\begin{thebibliography}{99}

%% \bibitem must have the following form:
   %\bibitem{key}...
%%

% \bibitem{}

%%%%%(No. 1)%%%
\bibitem{AKY1}
M. Adamaszek, A. Kozlowski  and K. Yamaguchi,
Spaces of algebraic and continuous maps between real algebraic varieties,
 Quart. J. Math. {\bf 62} (2011), 771--790.
%%%%(No.2)%%
\bibitem{AJ}M. F. Atiyah and J. D. S. Jones,
Topological aspects of Yang-Mills theory, Commun. Math. Phys.
{\bf 59} (1978), 97--118.
%%(No. 3)%%%
\bibitem{A}
M. F. Atiyah, Instantons in dimension two and four, Commn. Math. Phys.
{\bf 93} (1984), 437-451.
%%%%%%%%%(No. 4)%%%
\bibitem{BP}V. M. Buchstaber and T. E. Panov,
Torus actions and their applications in topology and combinatorics,
Univ. Lecture Note Series {\bf 24}, Amer. Math. Soc. Providence, 2002.
%%%(No. 5)%%
\bibitem{CMN}F. R. Cohen, J. C. Moore and J. A. Neisendorfer,
The double suspension and exponents of the homotopy groups of spheres,
Ann. Math., {\bf 110} (1979), 549--565.
%%(No. 6)%%
\bibitem{Cox1}D. A. Cox, The homogenous coordinate ring of a toric variety, 
J. Algebraic Geometry {\bf 4} (1995), 17-50. 
%%(No. 7)%%%
\bibitem{Cox2}D. A. Cox, The functor of a smooth toric variety,
Tohoku Math. J. {\bf 47} (1995), 251-262.
%%%%(No. 8)%%%
\bibitem{CLS}
D. A. Cox, J. B. Little and H. K. Schenck,
Toric varieties, Graduate Studies in Math. {\bf 124}, Amer. Math. Soc., 2011.
%%%%%%(No. 9)%%%
\bibitem{Gu1}
M. A. Guest, Instantons, rational maps and harmonic maps,
Methamatica Contemporanea {\bf 2} (1992), 113-155.
%%%%%%(No.10)%%%
\bibitem{Gu2}M. A. Guest,
The topology of the space of rational curves on a toric variety,
Acta Math. {\bf 174} (1995), 119--145.
%%%%%%(No. 11)%%
\bibitem{GKY1}
M.A. Guest, A. Kozlowski and K. Yamaguchi,
The topology of spaces of coprime polynomials,
Math. Z. {\bf 217} (1994), 435--446.
%%%%%%(No. 12)%%%
\bibitem{GKY2}
M. A. Guest, A. Kozlowski and K. Yamaguchi,
Spaces of polynomials with roots of bounded
multiplicity,
Fund. Math. {\bf 116} (1999), 93--117.
%%%%(No. 13)%%%%
\bibitem{Ka1}
S. Kallel, Divisor spaces on punctured Riemann surfaces,
Trans. Amer. Math. Soc. {\bf 350} (1998), 135--164.
%%(No.14)%%%%%
\bibitem{Ka2}
S. Kallel, Spaces of particles of manifolds and generalized
Poincar\'e dualities, Quart. J. Math. {\bf 52} (2001),45--70.
%%(No. 15)%%%%%%%
\bibitem{KOY1}A. Kozlowski, M. Ohno and K. Yamaguchi,
Spaces of algebraic maps from real projective spaces to toric varieties,
J. Math. Soc. Japan {\bf 68} (2016), 745-771
%%%(No. 16)%%%
\bibitem{KY4}A. Kozlowski and K. Yamaguchi,
Simplicial resolutions and spaces of algebraic maps between real
projective spaces, Topology Appl. {\bf 160} (2013), 87--98.
%%%(No. 17)%%%
\bibitem{KY6}A. Kozlowski and K. Yamaguchi, 
The homotopy type of spaces of coprime polynomials revisited, 
Topology Appl. {\bf 206} (2016), 284-304.
%%%(No. 18)%%%%
\bibitem{KY7}A. Kozlowski and K. Yamaguchi,
The homotopy type of spaces of polynomials with bounded multiplicity, Publ. RIMS. Kyoto Univ., {\bf 52} (2016), 297-308. 
%%%(No. 19)%%%%
\bibitem{Mo2}J. Mostovoy,
Spaces of rational maps and the Stone-Weierstrass Theorem, 
Topology {\bf 45} (2006), 281--293.
%%(No. 20)%%%
\bibitem{Mo3}J. Mostovoy,
Truncated simplicial resolutions and spaces of rational maps,
Quart. J. Math. {\bf 63} (2012), 181--187.
%%%%(No. 21)%%%
\bibitem{MV}J. Mostovoy and E. Munguia-Villanueva,
Spaces of morphisms from a projective space to a toric variety,
Rev. Colombiana Mat. {\bf 48} (2014), 41-53.
%%%%(No. 22)%%%
\bibitem{Pa1}T. E. Panov,
Geometric structures on moment-angle manifolds,
Russian Math. Surveys {\bf 68} (2013), 503--568.
%%%(No. 23)%%%
\bibitem{Se}G. B. Segal,
The topology of spaces of rational functions,
Acta Math. {\bf 143} (1979), 39--72.
%%%(No. 24)%%%%%%
\bibitem{Va}V. A. Vassiliev,
Complements of discriminants of smooth maps, 
Topology and Applications, 
Amer. Math. Soc.,
Translations of Math. Monographs \textbf{98},
1992 (revised edition 1994).
%%%(No. 25)%%%%%%
\bibitem{Va2}  V. A. Vassiliev, Topologia dopolneniy k diskriminantam, Fazis,  
Moskva 1997.
%%%%%%%%%%%%%%%%%%%
\end{thebibliography}

%% Authors are advised to submit their bibtex database files. They are
%% requested to list a bibtex style file in the manuscript if they do
%% not want to use model1-num-names.bst.

%% References without bibTeX database:

\end{document}